\documentclass[11pt,a4paper]{amsart}
\usepackage{mathrsfs}
\usepackage{amssymb,amscd,amsthm,cite,hyperref}
\input xy
\xyoption{all}

\usepackage{ifpdf}
\ifpdf 
  \usepackage[pdftex]{graphicx}
  \DeclareGraphicsExtensions{.pdf,.png,.jpg,.jpeg,.mps}
  \usepackage{pgf}
  \usepackage{tikz}
\else 
  \usepackage{graphicx}
  \DeclareGraphicsExtensions{.eps,.bmp}
  \usepackage{pgf}
  \usepackage{tikz}
  \usepackage{pstricks}
\fi

\def\A{\mathscr A}
\def\H{\mathscr H}

\def\ind{\operatorname{ind}}

\def\ord{\operatorname{ord}}
\def\re{\operatorname{Re}}
\def\res{\operatorname{Res}}

\newcommand{\Res}{\mathop{\rm Res}}
\newcommand{\im}{\mathop{\rm Im}}

\def\ch{\operatorname{ch}}
\def\sh{\operatorname{sh}}
\def\cth{\operatorname{cth}}
\def\ctg{\operatorname{ctg}}

\def\Mat{\operatorname{Mat}}
\def\tg{\operatorname{tg}}
\DeclareMathOperator\tr{Tr}

\DeclareMathOperator\Tr{Tr}

 \def\lra{\longrightarrow}

\def\dim{\operatorname{dim}}

\DeclareMathOperator\GL{GL}
\DeclareMathOperator\Mp{Mp}
\DeclareMathOperator\Sp{Sp}

\DeclareMathOperator\End{End}

\DeclareMathOperator\diag{diag}

\newtheorem{proposition}{Proposition}
\newtheorem{theorem}{Theorem}
\newtheorem{corollary}{Corollary}
\newtheorem{lemma}{Lemma}

\theoremstyle{definition}
\newtheorem{remark}{Remark}
\newtheorem{definition}{Definition}

\def\im{\operatorname{Im}}

\title[Local Index Formulae for the Affine Metaplectic Group]{
Local Index Formulae on Noncommutative Orbifolds and Equivariant Zeta Functions for the Affine Metaplectic Group}

\author{Anton Savin}
\address{A. Savin. Peoples' Friendship University of Russia (RUDN University),
6 Miklukho-Maklaya St, Moscow, 117198, Russia}
\email{a.yu.savin@gmail.com} 

\author{Elmar Schrohe} 
\address{E. Schrohe.  Leibniz University Hannover, Institute of Analysis, Welfengarten 1, 30167 Hannover, Germany }
\email{schrohe@math.uni-hannover.de} 

\date{}
\begin{document}

\maketitle

\begin{abstract} 
We consider the algebra $\A$ of bounded operators on $L^2(\mathbb{R}^n)$ generated by quantizations of isometric affine canonical transformations.
The algebra $\A$ includes as subalgebras 
noncommutative tori of all dimensions and toric orbifolds.
We define the spectral triple $(\A, \H, D)$  with $\H=L^2(\mathbb R^n, \Lambda(\mathbb R^n))$ and the Euler operator $D$, a first order differential operator of index $1$.   
We show that this spectral triple has simple dimension spectrum: For every operator $B$ in the algebra $\Psi(\A,\H,D)$ generated by the Shubin type pseudodifferential operators and the elements of $\A$, the zeta function $\zeta_B(z) = \Tr (B|D|^{-2z})$ has a meromorphic extension to $\mathbb C$ with at most simple poles.
Our main result then is an explicit algebraic expression for the Connes-Moscovici 
cocycle.  
As a corollary we obtain local index formulae for noncommutative tori and toric orbifolds. 
\\[.5ex]
\end{abstract}

\tableofcontents

\section{Introduction}

In this article we present a local index formula for a spectral triple associated with the affine metaplectic group. As special cases we obtain local index formulae for noncommutative tori of arbitrary dimension and noncommutative toric orbifolds.  We follow the noncommutative geometry approach laid out in the classical paper by Connes and Moscovici \cite{CoMo1}. 
 
The key notion is that  of  spectral triple. A spectral triple  $(\A,\H,D)$ consists of an algebra $\A$, a Hilbert space $\H$, and an unbounded operator   $D$ acting on $\H$. In addition, $\A$ acts on $\H$ by bounded operators, and the commutators $[D,a]$ are bounded for all $a\in \A$. A classical example  is the  spectral triple   
\begin{equation}
\label{dir1}
(C^\infty(M),L^2(M,S),D), 
\end{equation}
where $C^\infty(M)$ is the algebra of smooth functions on a closed smooth Riemannian spin
manifold $M$, $L^2(M,S)$ is the space of $L^2$-sections of the spinor bundle, while  $D$ is the Dirac operator on spinors. Under certain conditions,   spectral triples define  classes in  the Kasparov $K$-homology of $\A$ and one obtains the Chern--Connes character  
$$
\ch(\A,\H,D)\in HP^{*}(\A)
$$
in periodic cyclic cohomology of $\A$. The local index formula of Connes and Moscovici   \cite{CoMo1} expresses this class in terms of  periodic cyclic cocycles on $\A$, which are described in terms of regularized traces on $\A$. 
In the case of the Dirac spectral triple   \eqref{dir1} these regularized traces reduce to the celebrated Wodzicki residue  \cite{Wod2}, and the local formula of Connes and Moscovici gives the classical local index formula, see \cite{Pon4}.  Let us emphasize, however, that to obtain an explicit index formula  in a new situation using Connes' and Moscovici's formula, one has to study  these regularized traces and carry out their explicit computation. For applications of the Connes--Moscovici formula see~\cite{CoMo3,NeTu1,CheHu1,Dab1, SDLSV1,PoWa2,Les3}.
  
Let us now describe the spectral triple under consideration. Denote by $\Mp^c(n)$  the complex metaplectic group (see e.g. Leray \cite{Ler8}, \cite{Fol1,deG1,SaSc2} or Section~\ref{sec-mp}). 
One of many equivalent definitions of this group says that this is the group of all unitary operators acting on $L^2(\mathbb{R}^n)$ equal to quantizations of linear canonical transformations of the cotangent bundle $T^*\mathbb{R}^n$. More generally, if we consider  affine canonical transformations of $T^*\mathbb{R}^n$,  we obtain the affine complex metaplectic group. We set $\A$ to be the algebra of bounded operators on $L^2(\mathbb{R}^n)$ generated by quantizations of {\em isometric} affine canonical transformations. It can be shown that $\A$   has the following generators:
\begin{enumerate}
\item[1)]  Heisenberg-Weyl operators: $u(x)\mapsto e^{ikx-iak/2}u(x+a)$, where $a,k,x\in \mathbb{R}^n$; 
\item[2)]  Shift operators: $u(x)\mapsto u(g^{-1}x)$, where $g\in O(n)$ is an orthogonal matrix;
\item[3)]  Fractional Fourier transforms for $\varphi\in(0,\pi)$ (see e.g.~\cite{BuMa1} or Section \ref{sec-mp}): 
$$
u(x_1,x_2,...,x_n)\mapsto \sqrt{\frac{1-i\ctg \varphi}{2\pi}}\int 
   \exp\left(
        i\left(
         (x_1^2+y_1^2)\frac{\ctg \varphi}2-\frac{x_1y_1}{\sin\varphi} 
        \right) 
       \right) u(y_1,x_2,...,x_n)dy_1
$$       
\end{enumerate}
Generators of the form 1) are quantizations of shifts in $T^*\mathbb{R}^n$, those of the form 2) are quantizations of differentials of orthogonal transformations, while those in 3) are quantizations of rotations in the $(x_1,p_1)$ plane. This algebra includes as subalgebras noncommutative tori of all dimensions and toric orbifolds~\cite{Con1,FaWa1,Walt4,Walt2,Walt3,Walt1,BCHL1,ChLu1}.  
Moreover, we set $\H=L^2(\mathbb{R}^n,\Lambda(\mathbb{R}^n))$ and show that the elements in $\A$  naturally act on the differential forms.
Finally, our operator $D$ is the well-known Euler operator, a differential operator  of index one on $\mathbb{R}^n$, see e.g. Higson, Kasparov and Trout  \cite{HKT1}. 
 
Our first result asserts that this spectral triple satisfies the conditions in the local index theorem of Connes and Moscovici. Using the stationary phase approximation we show that the zeta functions
$$
\zeta_{a,b}(z)=\tr(ab|D|^{-2z}),
$$
where $a\in \A$ and $b$ is a pseudodifferential operator of Shubin type on $\mathbb{R}^n$, admit a meromorphic continuation to $\mathbb{C}$ with simple poles; in other words, the spectral triple has simple dimension spectrum.

Our second result is an explicit algebraic formula for the Connes--Moscovici cocycle. We express this cocycle as a sum of contributions over the fixed point sets of the canonical transformations. The computation reduces to obtaining (i) equivariant heat trace asymptotics for the quantum oscillator with respect to elements of the affine metaplectic group (bosonic part) and (ii) heat trace asymptotics for operators given by Clifford products acting on algebraic forms (fermionic part). ?To compute these asymptotics, we use the Mehler formula and Getzler's calculus.
Furthermore, we analyze the Connes--Moscovici periodic cyclic  cocycle and show that it is in fact a sum of cyclic cocycles localized at  conjugacy classes in the group of isometric affine canonical transformations (this group is isomorphic to the semidirect product $\mathbb{C}^n\rtimes U(n)$). 
As applications, we give  explicit index formulae for noncommutative tori and for noncommutative toric orbifolds.
It turns out that noncommutative tori correspond to choosing lattices in $\mathbb{C}^n$, while orbifolds correspond to finite groups acting on such lattices. 

As mentioned above, our algebra $\A$ is generated by quantized canonical transformations acting in $L^2(\mathbb{R}^n)$. Thus, this paper is part of our ongoing project to construct an index theory associated with groups of quantized canonical transformations. 
The articles \cite{SaSchSt4} and \cite{SaSch1} focused on operators on closed manifolds, 
see also Gorokhovky, de Kleijn and Nest \cite{GKN} for related work.
In the recent article \cite{SaSc2} we treated the algebra generated by the metaplectic operators and the Shubin type pseudodifferential operators on $\mathbb R^n$ and obtained an index formula using K-theory.  
Here, in contrast, we study the algebra generated by the affine metaplectic group, define a spectral triple,  and find explicit expressions for the Connes-Moscovici cocycles  associated with it. 
 
\subsection*{Historical notes and relation to previous work} 
The local index formula in  \cite{CoMo1} was extended to twisted spectral triples by Connes and Moscovici, \cite{CoMo2}, \cite{Mos2}.
A conceptually different approach to the local index formula was developed by Carey, Phillips, Rennie and Sukochev in
\cite{CPRS1,CPRS2};  it allowed them to derive the local index formula without the technical condition on the rapid decay of the zeta functions along vertical lines in $\mathbb C$ that Connes and Moscovici had imposed;  see also
 Higson \cite{Hig6} for another derivation.

As for concrete applications, Connes and Moscovici stated the formula for the case of the Dirac operator on a closed spin manifold \cite[Remark II.1]{CoMo1}. 
In \cite{Pon4}, Ponge derived the formula for the Connes-Moscovici 
cocycle for a Dirac spectral triple from his new proof of the local index formula. 
Chern and Hu \cite{CheHu1} and Azmi \cite{Azmi00} computed the corresponding 
expressions for equivariant Dirac operators via heat kernel techniques, however without verifying the technical assumptions  made in \cite{CoMo1}.   
A complete treatment of the equivariant case was eventually given by Ponge and Wang \cite{PoWa2}.  
In \cite{SDLSV1} and \cite{Dab1} van Suijlekom, D\c{a}browski, Landi, Sitarz and Varilly obtained a local index formula for the quantum $SU(2)$.  
For a recent survey on index theory and noncommutative geometry see Gorokhovsky and van Erp \cite{GoErp1}.

Noncommutative tori and noncommutative orbifolds are central and actively researched objects in noncommutative geometry. The local index formula for 
the noncommutative two torus can be found in \cite{Con1}. This was extended recently by Fathizadeh, Luef and Tao \cite{FaLuTa}, who established a local formula for the index of the Dirac operator of a twisted spectral triple on the two torus. However, the local index is given as a number without considering the Connes-Moscovici cocycle. 
Chakraborty and Luef \cite{ChLu1}, building on and partly generalizing work by Echterhoff, L\"uck, Phillips and Walters \cite{Walt2} and Walters \cite{Walt3,Walt1}, 
studied $n$-dimensional noncommutative tori with an action of a finite cyclic group.
They used metaplectic representations and obtained structural and K-theoretical results, but no local index formulae. 
There also is a series of recent articles by Ponge and collaborators concerning pseudodifferential operators and differential geometric objects  on noncommutative two tori, see e.g. \cite{Pon20}. 
In a different vein, Mathai and Rosenberg in \cite{MaRo20} showed a Riemann--Roch theorem and a Hodge theorem for $n$-dimensional complex noncommutative  tori. The proof is via deformation to the commutative case without the use of local index formulae. 
Another interesting development is the pseudodifferential calculus on quantum Euclidean spaces developed by Gao, Junge and McDonald \cite{GaJuDo}, which they use to derive a local index formula for this situation. 
It is not unlikely that this calculus can also be used in the present context. For the purposes of this article, however, the Shubin pseudodifferential calculus combined with the metaplectic operators is a much simpler and completely adequate tool.

\subsection*{Structure of the article} We start by recalling the local index formula by Connes and Moscovici in  Section~\ref{sec1}. We state their result both in terms of zeta functions and heat trace asymptotics. 
In two subsequent sections  we recall necessary information about the metaplectic group and pseudodifferential operators of Shubin type.  Then  we obtain  in Section~\ref{sec2} the local index formula in the one-dimensional case. Section~\ref{sec3} is central to our paper: Here we define our spectral triple in $\mathbb{R}^n$, show that it satisfies the conditions of Connes' and Moscovici's theorem (Theorem~\ref{th34})
and give explicit formulae for the Connes--Moscovici 
cocycle (Theorem~\ref{th-main3}). A decomposition of the 
Connes--Moscovici periodic cyclic cocycle into a sum of cyclic cocycles localized at conjugacy classes in $\mathbb{C}^n\rtimes U(n)$ is obtained in Section~\ref{sec4}, while examples are considered in Section~\ref{sec5}.  In Section~\ref{sec:zeta}, we prove all the necessary results about equivariant zeta functions for Shubin type pseudodifferential operators and metaplectic operators. 

\subsection*{Acknowledgement.} ES thanks Gerd Grubb and Jens Kaad for helpful remarks.
The work of the first author was supported by   
RFBR, project number 21-51-12006; that of the second by DFG through project SCHR 319/8-1. 
We thank the referees for useful suggestions. 

\section{The Local Index Formula of Connes and Moscovici}\label{sec1}

\subsection*{The Chern-Connes character.}

Let $(\A,\H,D)$ be a spectral triple. Here
\begin{itemize} 
\item $\A$ is an algebra; 
\item $\H=\H_0\oplus \H_1$ is a graded Hilbert space with a representation of $\A$ on it by bounded even operators;
\item $D$ is an odd self-adjoint operator on $\H$. 
It is assumed that $D$ is local: $[D,a]$ is bounded for all $a\in \A$ and $(1+D^2)^{-1}$ is compact.
\end{itemize}

Suppose, moreover, that the  spectral triple is $p$-summable, i.e. 
$$
(1+D^2)^{-1/2}\in \mathcal{L}^p(\H)
$$
where $\mathcal{L}^p(\H)=\{T \;|\; T \text{ is compact and }\tr|T|^p<\infty\}$ is the  Schatten von Neumann ideal.
Then one defines the Chern--Connes character of the spectral triple in periodic cyclic cohomology, see  \cite{Con1}, \cite{Hig6}:
$$
\ch(\A,\H,D)\in HP^{ev}(\A).
$$
It contains information about the analytic indices of twisted operators. More precisely, given a projection $P\in \Mat_N(\A)$ over $\A$, we have
$$
\ind (P (D\otimes 1_N): P\H_0^N \lra P\H_1^N)=\langle \ch (\A,\H,D),[P]\rangle,
$$
where $[P]\in K_0(\A)$ is the class of the projection in $K$-theory, while 
$$
 \langle\cdot ,\cdot \rangle: HP^{ev}(\A)\times K_0(\A)\longrightarrow \mathbb{C}
$$ 
stands for the pairing of periodic cyclic cohomology with $K$-theory. 
Let us recall the definition of this pairing. Given a periodic cyclic cocycle $\varphi=(\varphi_{2k})_{k=0,1,...,n}$ over $\A$ and a projection $p\in\Mat_N(\A)$, we have
\begin{equation}\label{spar}
 \langle\varphi,p\rangle=(\varphi_{0}\# \Tr)(p)+ \sum_{k\ge
1} (-1)^k\frac{(2k)!}{k!}(\varphi_{2k}\# \Tr)(p-1/2,p,...,p),
\end{equation}
where 
\begin{equation}\label{spar1}
(\varphi_{2k}\# \Tr)(m_0\otimes a_0,...,m_{2k}\otimes
a_{2k})=\Tr(m_0...m_{2k})\varphi_{2k}(a_0,...,a_{2k}),\quad a_k\in \A,m_k\in \Mat_N(\mathbb C)
\end{equation}
is a $2k+1$-linear functional on $\Mat_N(\A)$.  

\subsection*{The local index formula.}

Connes and Moscovici \cite{CoMo1} (see also Higson \cite{Hig6}) proved 
that the  class   $\ch(\A,\H,D)$ contains a special representative, 
which we call the Connes--Moscovici cocycle. To state their result, we introduce several notions. Let us assume for simplicity that $D$ is invertible (for the noninvertible case see \cite{CoMo1}).

The spectral triple $(\A,\H,D)$ is supposed to be regular, see Definition 3.14 in \cite{Hig6}, i.e. for every $a\in \A$,   
$a$ and $[D,a]$ are in the domains of all iterated commutators 
$$
 [|D|,\cdot],\quad [|D|,[|D|,\cdot]],\quad \ldots. 
$$ 

Given a regular spectral triple, one defines the algebra $\Psi(\A,\H,D)$ as the smallest algebra of linear operators in $\H^\infty=\cap_{j\ge 1}{\rm Dom} |D|^j$ that  contains $\A$ and $[D,\A]$ and is closed under taking commutators with $D^2$: $B\in \Psi(\A,\H,D)$ implies that $[D^2,B]\in \Psi(\A,\H,D)$.

Given $B\in \Psi(\A,\H,D)$, we introduce the zeta function $\zeta_B$ by 
$$
 \zeta_B(z)=\tr(B|D|^{-2z}),
$$
which is defined and holomorphic for $\re z$ sufficiently large.

We say that $(\A,\H,D)$ has simple dimension spectrum, if there is a discrete set $F\subset \mathbb{C}$ such that 
$\zeta_B(z)$ extends meromorphically to $\mathbb{C}$ with at most simple poles in the set $F+\ord B$ for all $B\in \Psi(\A,\H,D)$.

\begin{theorem}\label{th0}
{\bf (Connes-Moscovici)}
Suppose that the spectral triple has simple dimension spectrum. 
Then the Chern--Connes character  $\ch(\A,\H,D)\in HP^{ev}(\A)$ in periodic cyclic cohomology has a representative $(\Psi_0,\Psi_2,\Psi_4,...,\Psi_{2k},...)$, where 
\begin{equation}\label{ht1a}
\Psi_{2k}(a_0,a_1,...,a_{2k}) =\sum_{\alpha}c_{k,\alpha}\res_{z=0}
\tr_s \left(a_0[D,a_1]^{[\alpha_1]}... [D,a_{2k}]^{[\alpha_{2k}]}|D|^{-2(|\alpha|+k+z)}\right),\quad k\ge 1,
\end{equation}
$\alpha=(\alpha_1,\alpha_2,...,\alpha_{2k})$ is a multi-index, $B^{[j]}$ stands for the $j$-th  iterated commutator of the operator $B$ with $D^2$,  and
$$
c_{k,\alpha}=(-1)^{|\alpha|}\frac{\Gamma(|\alpha|+k)}{\alpha!(\alpha_1+1)...(\alpha_1+...+\alpha_{2k}+2k)},
$$
while, see \cite[Remark 5.7]{Hig6}, 
\begin{equation}\label{ht1b}
\Psi_0(a_0)=\res_{z=0}z^{-1}
\tr_s \left(a_0 |D|^{-2 z}\right).
\end{equation}
\end{theorem}

\begin{remark}
\label{decay}
Connes and Moscovici additionally required the zeta functions $\zeta_B$ to have rapid decay 
along vertical lines. It was shown by Carey, Phillips, Rennie and Sukochev 
\cite{CPRS1,CPRS2} that this assumption is not needed, see also Higson \cite{Hig6}. 
In the case at hand, rapid decay can be established directly with the help of the weakly parametric calculus of Grubb and Seeley, \cite{GS1}.  
\end{remark}

\subsection*{The local index formula and heat trace asymptotics.}
The following proposition will enable us to apply techniques in local index theory based on heat trace asymptotics.

\begin{proposition}\label{prop2}
Let the assumptions in Theorem~\ref{th0} be satisfied for the spectral triple $(\A,\H,D)$. Suppose in addition that for an operator  $B\in \Psi(\A,\H,D)$ the heat trace $\tr(Be^{-tD^2})$ exists, is a continuous function of $t>0$, is $O(t^{-\infty})$ for $t\to\infty$ and  has an asymptotic expansion 
\begin{equation}\label{heat-eq4}
\tr(Be^{-tD^2}) \sim \sum_{k=0}^\infty a_{m_k} t^{m_k}\quad \text{as }t\to 0^+
\end{equation}
for a sequence of real numbers $m_k\nearrow \infty$. Then  
\begin{equation}\label{res3}
\Res_{z=0} \tr(B|D|^{-2(m+z)})=\frac{a_{-m}}{\Gamma(m)},\qquad \text{whenever }m>0.
\end{equation}
Moreover, for $m=0$,
\begin{equation}\label{res3a}
\Res_{z=0} z^{-1}\tr(B|D|^{-2 z })= a_{0} .
\end{equation} 
\end{proposition}

\begin{proof}
Let $m>0$. We use the Mellin  transform $M_{t\to z}$  and obtain for $\re z>0$: 
\begin{equation}
\label{mellin2}
|D|^{-2z}=\frac 1{\Gamma(z)}\int_0^\infty t^{z-1} e^{-tD^2}dt, \quad\text{and hence }\quad 
\zeta_B(z)=\frac{1}{\Gamma(z)} M_{t\to z}(\tr(Be^{-tD^2})).
\end{equation}
The Mellin transform is well defined by our assumptions for large $\re z$. Hence
\begin{eqnarray*}\lefteqn{
 \tr(B|D|^{-2(m+z)})=
\frac 1{\Gamma(z+m)}\int_0^\infty t^{z+m-1} \tr(B e^{-tD^2})dt }\\
&\equiv&\
\frac 1{\Gamma(z+m)}\int_0^1 t^{z+m-1} \tr(B e^{-tD^2})dt  
\equiv
\frac 1{\Gamma(z+m)}\sum_{k\;:\; m+m_k< 1} a_{m_k}\int_0^1 t^{z+m+m_k-1} dt , 
\end{eqnarray*}
where $\equiv$ means equality modulo  functions holomorphic for $\re z>-1$. We then conclude that 
\begin{eqnarray*}\Res_{z=0} \tr(B|D|^{-2(m+z)})&=&\Res_{z=0}
\frac 1{\Gamma(z+m)}\sum_{k\;:\; m+m_k< 1} \left. \frac{a_{m_k}t^{z+m+m_k}}{z+m+m_k}\right|^1_0  
=\frac{a_{-m}}{\Gamma(m)} .
\end{eqnarray*}
The case $m=0$ is considered similarly. 
\end{proof}

\begin{remark}
It is possible to use properties of the heat trace and properties of the Mellin transform (see e.g. \cite[Theorems 3 and 4]{FGD1}) in order to obtain the properties of zeta functions listed in Theorem~\ref{th0}. For instance, if we require that $\tr(Be^{-tD^2})$ is smooth, $O(t^{-\infty})$ for $t\to\infty$, and has an asymptotic expansion \eqref{heat-eq4} as $t\to 0$,
then this implies that $\zeta_B(z)$ has a meromorphic continuation to $\mathbb{C}$ and  rapid decay on vertical lines.
\end{remark}

Suppose that all the operators $B= a_0[D,a_1]^{[\alpha_1]} \ldots [D,a_{2k}]^{[\alpha_{2k}]}$ in \eqref{ht1a}
satisfy the assumptions in Proposition~\ref{prop2}. Then we apply Proposition~\ref{prop2} and express the 
Connes--Moscovici cocycle $\{\Psi_{2k}\}$ in terms of heat invariants and get: 
\begin{equation}\label{ht1}
\Psi_{2k}(a_0,a_1,...,a_{2k}) =\sum_{\alpha}d_{k,\alpha} \times \left\{\text{ finite part of } t^{|\alpha|+k}
\tr_s \left(a_0[D,a_1]^{[\alpha_1]}... [D,a_{2k}]^{[\alpha_{2k}]}e^{-tD^2}\right)\right\},
\end{equation}
for all $k\ge 0$, where   
$$
d_{k,\alpha}=\frac{ (-1)^{|\alpha|}}{\alpha!(\alpha_1+1)...(\alpha_1+...+\alpha_{2k}+2k)}.
$$


\section{The Metaplectic Group}\label{sec-mp}

Let us recall a few facts about the  symplectic  and metaplectic groups from \cite{Ler8,deG1,SaSc2}.

\subsection*{The symplectic and the metaplectic groups.}

The metaplectic group $\Mp(n)\subset \mathcal{B}L^2(\mathbb{R}^n)$ is the group generated by unitary operators of the form 
$$
\exp(-i\widehat H)\in \Mp(n),
$$
where $\widehat H$ is the Weyl quantization of a homogeneous real quadratic Hamiltonian
$H(x,p)$, $(x,p)\in T^*\mathbb{R}^n$. In its turn, the complex metaplectic group $\Mp^c(n)$ is similarly
generated by unitaries associated with Hamiltonians $H(x,p)+\lambda$, where $H(x,p)$ is as above, while $\lambda$ is a real constant. Elements of the metaplectic group are called metaplectic operators.

The symplectic group $\Sp(n)\subset \GL(2n,\mathbb{R}) $ is the group of linear canonical transformations of $T^*\mathbb{R}^n\simeq \mathbb{R}^{2n}$, i.e., linear transformations that preserve the symplectic form $dx\wedge dp$. 
The symplectic group is generated by the   canonical transformations equal to the evolution operator for time $t=1$ of the Hamiltonian system 
$$
 \dot x=H_p,\quad \dot p=-H_x,
$$ 
where $H(x,p)$ is a homogeneous real quadratic Hamiltonian as above.

There is a natural  projection $\pi:\Mp(n)\to\Sp(n)$ that takes a metaplectic operator to the corresponding canonical transformation. This projection is  a nontrivial double covering of $\Sp(n)$.    Thus, one can not represent  elements of $\Sp(n)$ unambigously by metaplectic operators. However, it turns out that one can define  a representation of the   subgroup of isometric linear canonical transformations  by operators in the complex metaplectic group. Let us describe this representation.

\subsection*{Isometric linear canonical transformations and their quantization.} 
Consider the maximal compact subgroup $\Sp(n)\cap O(2n)$ of isometric linear canonical transformations. 
It is well known that this intersection coincides with the group $U(n)$  of unitary transformations of $T^*\mathbb{R}^n$ if we introduce the complex structure on $T^*\mathbb{R}^n\simeq \mathbb{C}^n$ via $(x,p)\mapsto z= p+ix$, see~\cite{Arn1}.

Recall that the unitary group is generated by the matrices $\exp( B+iA)$, where $A$ is a symmetric real matrix, while  $B$ is a  skew-symmetric real matrix.
\begin{proposition}
The following mapping is a well-defined  homomorphism of groups
\begin{equation}\label{secc1}
  \begin{array}{ccc}
    R: U(n) & \longrightarrow & \Mp^c(n)\vspace{2mm}\\
     g=\exp( B+iA) & \longmapsto & R_g= \exp(-i\widehat{H})\exp(i\tr A/2) ,
  \end{array}
\end{equation}
where $\widehat{H}$ is the Weyl quantization of the Hamiltonian
$$
 H(x,p)=\frac 1 2(x,p)\left(
                  \begin{array}{cc}
                    A & -B\\
                    B & A\end{array}\right)
\left(  \begin{array}{c}   x\\  p\end{array}\right).
$$  
\end{proposition}
Since $U(n)$ is generated by the subgroups $O(n)$ and  $U(1)=\{{\rm diag}(z,1,\ldots,1)\;|\; |z|=1\}$ (see e.g. \cite[Lemma~1]{SaSc2}), it follows that the homomorphism~\eqref{secc1} is characterized by the properties:
\begin{itemize}
\item $R_gu(x)=u(g^{-1}x)$, if $g\in O(n)\subset U(n)$; in this case $R_g$  is the shift operator for an orthogonal matrix $g$
\item $R_gu(x)=e^{i\varphi(1/2-H_1)}u(x)$, if $g={\rm diag}(e^{i\varphi},1,...,1),$ where $H_1=\frac12(x_1^2-\partial_{x_1}^2)$. In this case, the operator $R_g$ is called the fractional Fourier transform with respect to $x_1$ and for $\varphi\in (0,\pi)$ is equal to (see \cite[Corollary 4.2]{H95})
$$
R_gu(x)=\sqrt{\frac{1-i\ctg \varphi}{2\pi}}\int 
   \exp\left(
        i\left(
         (x_1^2+y_1^2)\frac{\ctg \varphi}2-\frac{x_1y_1}{\sin\varphi} 
        \right) 
       \right) u(y_1,x_2,...,x_n)dy_1. 
$$
\end{itemize} 

\section{Shubin Type Pseudodifferential Operators}\label{sec-shubin}
A smooth function $a=a(x,p)$ on $T^*\mathbb R^n$ is a
pseudodifferential symbol (of Shubin type) of order $m\in \mathbb R$, provided its derivatives satisfy the estimates 
$$|D^\alpha_p D^\beta_x a(x,p) |\le c_{\alpha,\beta}(1+|x|+|p|)^{m- |\alpha|- |\beta|}$$
for all multi-indices $\alpha$, $\beta$, with suitable constants $c_{\alpha, \beta}$. 
In this article, we only work with classical symbols where $a$ admits an asymptotic expansion $a\sim \sum_{j=0}^\infty
a_{m-j}$. Here, each $a_{m-j}$ is a symbol of order $m-j$, which is (positively) homogeneous in $(x,p)$ for $|x,p|\ge 1$.

To a symbol $a$ as above we associate the operator $\mathop{\rm op}(a) $ on the Schwartz space $\mathcal S(\mathbb R^n)$,  defined by 
$$\mathop{\rm op}(a)u(x) =(2\pi)^{-n/2} \int e^{ix\cdot p}a(x,p) \widehat u(p) \, dp,$$
where $\widehat u(p) =(2\pi)^{-n/2} \int e^{-ix\cdot p} u(x) \, dx$ is the Fourier 
transform of $u$. Alternatively, we have the Weyl quantization ${\rm op}^w (a)$ 
of $a$ defined by 
$${\rm op}^w (a)(x) = (2\pi)^{-n} \iint e^{i(x-y)\cdot p} a\Big(\frac{x+y}2,p\Big)u(y) \, dydp.$$

The principal symbol $\sigma(A)$ of the operator $A={\rm op}(a)$ is defined as the homogeneous extension of the leading term $a_m$ to $T^*\mathbb R^n\setminus \{0\}$. 

A full calculus for Shubin type pseudodifferential operators, i.e. pseudodifferential operators with such symbols, has been developed in \cite[Chapter IV]{Shu1}. 
We write $\Psi^m(\mathbb R^n)$ for the space of all Shubin type pseudodifferential operators of order $\le m$ and $\Psi (\mathbb R^n)$ for the algebra of all these operators.

A fact we need in several places is that the elements of $\Psi^0(\mathbb R^n)$ extend to bounded operators on $L^2(\mathbb R^n)$ and those of  $\Psi^m(\mathbb R^n)$ 
to trace class operators on $L^2(\mathbb R^n)$ provided $m<-2n$. This is shown in \cite[Theorem 24.3 and Proposition 27.2]{Shu1}. 

Moreover, a Egorov theorem holds: Given an element $S\in \Mp(n)$ and 
$A={\rm op}^w a$ a Weyl-quantized Shubin type pseudodifferential operator  with symbol $a$, then $S^{-1}AS$
is the Weyl-quantized Shubin type pseudodifferential operator with symbol $a\circ \pi(S)$, where $\pi(S)\in \Sp(n)$ is the canonical transformation associated with $S$, see 
\cite[Theorem 7.13]{deG1}. As the principal symbol of the Weyl-quantized operator ${\rm op}^w(a)$ coincides with that of ${\rm op}(a)$, we find in particular that 
$$\sigma (S^{-1}AS) = \sigma(A)\circ \pi(S).$$ 


\section{Operators on $\mathbb{R}$}\label{sec2}
We start with the case  $n=1$, as it is simpler and the results will be useful later on.

\subsection*{The Euler operator.}
We introduce 
\begin{equation}\label{eq-d2}
D=\frac 1{\sqrt 2}  \left(
  \begin{array}{cc}
  0 & x-\partial_x \\
  x+\partial_x & 0
 \end{array}
 \right): 
\mathcal{S}(\mathbb{R},\mathbb{C}^2)\to  \mathcal{S}(\mathbb{R},\mathbb{C}^2),\quad
\text{so that} \quad D^2=\left(
  \begin{array}{cc}
    H -\frac 1 2 & 0 \\
  0 &   H +\frac 1 2 
   \end{array}
 \right),
\end{equation}
where  $H =\frac12(x ^2-\partial_{x }^2)$.  

\subsection*{Heisenberg-Weyl operators.}For $a,k\in\mathbb R$ we define the Heisenberg-Weyl operators $T_{a,k}$ on $L^2(\mathbb R)$ by
$$
T_{a,k} u(x)=e^{ikx-iak/2}u(x-a).
$$
We shall also write $T_{a,k}=T_z$, where $z=a-ik$. These operators generate the Heisenberg group, and we have the product formula
$$
T_{z_1}T_{z_2}=e^{-i\im  z_1\overline{z_2}/2}T_{z_1+z_2}.
$$
We extend the action of the Heisenberg-Weyl operator $T_z$ 
to $\mathcal{S}(\mathbb{R},\mathbb{C}^2)$  by
$(u,v)\mapsto (T_zu,T_zv)$, denoting it by $T_z$ as in the scalar case. Then the following
commutation relations are true
\begin{equation}\label{eq-d3}
[D, T_z]=\frac 1{\sqrt 2}  \left(
  \begin{array}{cc}
  0 & z \\
  \overline z  & 0
 \end{array}
 \right)T_z,
\end{equation}
$$
{}[H,T_z]=\frac{1}{2}\left((k^2-a^2+2ax)T_z-2ik T_z\partial_x\right)=T_z\cdot (\text{operator of order 1}).
$$

\subsection*{Metaplectic operators (fractional Fourier transforms).}
Consider the representation
\begin{equation}\label{u1}
 \begin{array}{ccc}
    R: U(1) & \lra & \Mp^c(1)\vspace{2mm}\\
  e^{i\varphi} & \longmapsto & R_\varphi=e^{i(1/2-  H)\varphi}.
 \end{array}
\end{equation}
We obtain the commutation relation:
 \begin{equation}
\label{cr12q}
R_\varphi T_z R_{\varphi}^{-1}=  T_{e^{i\varphi}z}.
\end{equation}
We also extend the action of the metaplectic operator $R_g$  to $\mathcal{S}(\mathbb{R},\mathbb{C}^2)$  via
$$ 
 \mathbf{R}_\varphi  (u,v)  =(    R_\varphi u, e^{-i\varphi}R_\varphi v).
$$ 
It turns out that $D$ is $U(1)$-equivariant (see \cite{SaSc2} for the proof), i.e. 
\begin{equation}\label{eq-equiv2}
 \mathbf{R}_\varphi D \mathbf{R}_\varphi^{-1}=D.
\end{equation}

\subsection*{Heat asymptotics.}

\begin{proposition}\label{prop1}
We have the following asymptotics as $t\to 0^+:$
\begin{equation}\label{heat1}
\tr(T_zR_{\varphi}e^{-t H})=\left\{
 \begin{array}{ll}
  \displaystyle  O(t^{+\infty}) & \text{ if }\varphi=0, z\ne 0; \vspace{2mm}\\
  \displaystyle   \frac 1{\sqrt 2\sqrt{\ch t-1}}=\frac  1 t +O(1)& \text{ if }\varphi=0, z=0;  \vspace{2mm}\\
   \displaystyle  \frac 1{1-e^{-i\varphi}} 
   \exp
\left(\frac i4(a^2+k^2)\ctg(\varphi/2) \right) +O(t)& \text{ if }\varphi \in (0,2\pi).\\
 \end{array}
\right.
\end{equation}
Here $T_zR_{\varphi}e^{-t H}$ is treated as an operator on $L^2(\mathbb{R})$.
\end{proposition}

\begin{proof}
1. In case $\varphi= 0$,  Mehler's formula for the heat kernel
$$
e^{-t H}(x,y)=\frac 1{\sqrt{2\pi \sh t}}\exp\left(-\cth t\frac{x^2+y^2}2+\frac 1{\sh t}xy\right)
$$
shows that 
$$
T_ze^{-t H}(x,y)=\frac {e^{i(kx-ak/2)}}{\sqrt{2\pi \sh t}}\exp\left(-\cth t\frac{(x-a)^2+y^2}2+\frac 1{\sh t}(x-a)y\right).
$$
Hence, the value of the kernel at the diagonal is equal to
$$
T_ze^{-t H}(x,x)=\frac {e^{-iak/2}}{\sqrt{2\pi \sh t}}\exp\left(- x^2\left( \frac{\ch t-1}{\sh t}\right)
+x\left(ik+a\frac{\ch t-1 }{\sh t}\right)-a^2\frac{\ch t}{2\sh t} \right) 
$$
and we obtain\footnote{We use the formula for the Gaussian integration
$$
\int_{\mathbb{R}} e^{-ax^2+bx+c}dx=\sqrt{ \frac \pi a}e^{b^2/4a+c}.
$$}
\begin{equation}\label{heat2}
  \tr(T_z e^{-t H})=\int_\mathbb{R}  T_ze^{-t H}(x,x) dx 
=\frac 1{\sqrt 2}\frac{1}{\sqrt{\ch t -1}} \exp\left(
  -a^2\frac{\ch t+1}{4\sh t}-k^2\frac{\sh t}{4(\ch t -1)}
 \right).
\end{equation}
This readily gives us  the first two lines in \eqref{heat1}.

2. In case $\varphi\ne 0$ we have
$$
T_zR_{\varphi}e^{-t H}=T_z e^{i(1/2- H)\varphi} e^{-t H}=
e^{i\varphi/2}T_ze^{-(t+i\varphi) H}.
$$
We can compute the trace of this operator by replacing $t$ by $t+i\varphi$ in \eqref{heat2}. Then
\begin{equation}\label{heat3}
  \lim_{t\to 0+}\tr(T_z R_{\varphi}e^{-t H}) 
=\frac{e^{i\varphi/2}}{ i\sqrt 2\sqrt{1-\cos\varphi}} \exp
\left(\frac i 4\left(   a^2\frac{1+\cos\varphi}{\sin\varphi}+k^2\frac{\sin\varphi}{1-\cos\varphi} \right)\right),
\end{equation}
where $\varphi\in (0,2\pi)$. Here the argument of the square root is computed using the Taylor expansion at $t=0,\varphi=0$: $\sqrt{\ch (t+i\varphi) -1}\sim \sqrt{ (t+i\varphi)^2/2}=(t+i\varphi)/\sqrt{2}$.  
This expression is equal to the last line in \eqref{heat1}.
\end{proof}

\subsection*{The local index formula.} 
We denote by  $\A$ the algebra generated by the operators $T_z$ and $\mathbf R_g$, $z\in \mathbb C$, $g\in U(1)$, acting on  the Hilbert space $\H= L^2(\mathbb R,\mathbb C^2)$.
Let us compute the Connes--Moscovici   cocycle   of the spectral triple $(\A, \H, D)$ 
defined from the operator $D$. 
It follows from the above commutation relations that 
$$ a_0[D,a_1]^{[\alpha_1]}... [D,a_{2k}]^{[\alpha_{2k}]}|D|^{-2(|\alpha|+k+z)}$$
is an operator of order $\le
|\alpha|-2(|\alpha|+k+\re z)=-|\alpha|-2k-2\re z$.
As operators of order $<-2$ are of trace class, there are only two possibilities to obtain 
a nontrivial contribution to the Connes-Moscovici local index formula in \eqref{ht1a}/\eqref{ht1b} 
namely a) $k=0$  
and  b) $k=1,$ $|\alpha|=0$. This will be important for the proof of the following theorem.

\begin{theorem}
The component $\Psi_0$ of the Connes-Moscovici cocycle is equal to
\begin{equation}\label{fi0}
\Psi_0(T_z\mathbf{R}_\varphi)=\left\{
 \begin{array}{cl}
    0, & \text{ if } \varphi=0,z\ne 0,\\
    1, & \text{ if } \varphi=0,z=0,\\
    \exp \left(\frac i 4(a^2+k^2)\ctg(\varphi/2) \right),  & \text{ if } \varphi\ne 0. \\
 \end{array}
\right.
\end{equation}
Here $z=a-ik$.
The component $\Psi_2$ of the Connes-Moscovici cocycle is equal to
\begin{equation}\label{fi2}
\Psi_2(T_{z_0}\mathbf{R}_{\varphi_0},T_{z_1}\mathbf{R}_{\varphi_1},T_{z_2}\mathbf{R}_{\varphi_2})=\left\{
 \begin{array}{cl}
    0 & \text{ if } \varphi_0+\varphi_1+\varphi_2\ne 0  \\
 & \text{ or } z_0+e^{i\varphi_0}z_1+e^{i(\varphi_0+\varphi_1)}z_2\ne 0,\\
   \displaystyle \frac{e^{i\varepsilon}}4( z_1\overline z_2e^{-i\varphi_1}-\overline z_1z_2e^{i\varphi_1}) & \text{ otherwise},
 \end{array}
\right.
\end{equation}
where $ \varepsilon= \im(e^{i\varphi_0}z_1\overline{z}_0+e^{i\varphi_1}z_2\overline{z}_1+e^{i(\varphi_0+\varphi_1)}z_2\overline{z}_0).$
\end{theorem}

\begin{proof}
We saw that there is no contribution to the Connes-Moscovici cocycle from terms with 
$\alpha\not=0$. 
  
1. In order to show   \eqref{fi0}, we note that 
\begin{eqnarray*}%
\tr_s (T_z\mathbf{R}_\varphi e^{-tD^2})
&=&
\tr
\left(\begin{array}{cc}
1 & 0\\0 & -1
\end{array}
\right)
T_z
\left(\begin{array}{cc}
R_{\varphi} & 0\\0 & R_{\varphi}e^{-i\varphi}
\end{array}
\right)
\left(
  \begin{array}{cc}
  e^{-t(  H -\frac 1 2)} & 0 \\
  0 & e^{-t(  H +\frac 1 2)} 
   \end{array}
 \right)\\
&=&(e^{t/2}-e^{-i\varphi}e^{-t/2})\tr(T_zR_\varphi e^{-t  H}).
\end{eqnarray*}
According to \eqref{ht1b} and \eqref{res3a} we have to compute the constant term as  $t\to0^+$.
Substituting  the heat asymptotics from Proposition~\ref{prop1} gives precisely \eqref{fi0}.

2. Next let us prove \eqref{fi2}.  A direct computation using \eqref{cr12q}, \eqref{eq-d2}, \eqref{eq-d3}, \eqref{eq-equiv2} shows that
\begin{eqnarray}
\label{fi2a}
\lefteqn{\tr_s (T_{z_0}\mathbf{R}_{\varphi_0}[D,T_{z_1}\mathbf{R}_{\varphi_1}][D,T_{z_2}\mathbf{R}_{\varphi_2}]e^{-tD^2})\nonumber}\\
&=& \frac{1}{2}( z_1\overline z_2e^{-i\varphi_1}e^{t/2}-\overline z_1 z_2 e^{-i(\varphi_0+\varphi_2)}e^{-t/2})\tr (T_{z_0}R_{\varphi_0}T_{z_1}R_{\varphi_1}T_{z_2}R_{\varphi_2}e^{-t  H}))\nonumber\\
&=& \frac{1}{2}( z_1\overline z_2e^{-i\varphi_1}e^{t/2}-\overline z_1 z_2 e^{-i(\varphi_0+\varphi_2)}e^{-t/2})\tr (e^{i\varepsilon'}T_{w}R_{\varphi_0+\varphi_1+\varphi_2}e^{-t  H}),
\end{eqnarray}
where 
$$
 w=z_0+e^{ i\varphi_0}z_1+e^{ i(\varphi_0+\varphi_1)}z_2, \quad
  e^{i\varepsilon'}Id=T_{z_0} T_{e^{ i\varphi_0}z_1} T_{e^{ i(\varphi_0+\varphi_1)}z_2}  T^{-1}_w.
$$ 
By \eqref{ht1a} and \eqref{res3},  
$
\Psi_2(T_{z_0}\mathbf{R}_{\varphi_0},T_{z_1}\mathbf{R}_{\varphi_1},T_{z_2}\mathbf{R}_{\varphi_2})$
equals $1/2$ times the coefficient of $t^{-1}$ in the asymptotics of \eqref{fi2a} as $t\to 0^+$.
By Proposition~\ref{prop1}, this coefficient is zero if either $\varphi_0+\varphi_1+\varphi_2\ne 0 $  or $w=z_0+e^{ i\varphi_0}z_1+e^{ i(\varphi_0+\varphi_1)}z_2\ne 0$. Otherwise, we obtain
\begin{equation}
\label{eq-asymp5}
\tr_s (T_{z_0}\mathbf{R}_{\varphi_0}[D,T_{z_1}\mathbf{R}_{\varphi_1}][D,T_{z_2}\mathbf{R}_{\varphi_2}]e^{-tD^2})= \frac{1}{2}(z_1\overline z_2e^{ -i\varphi_1}-\overline z_1z_2e^{i\varphi_1})e^{i\varepsilon'}t^{-1}+O(1), 
\end{equation}
where 
\begin{multline*}
e^{i\varepsilon'}Id=T_{z_0} T_{e^{ i\varphi_0}z_1} T_{e^{ i(\varphi_0+\varphi_1)}z_2}=
e^{ i\im(\overline{z_0}e^{ i\varphi_0}z_1 )}T_{z_0+e^{ i\varphi_0}z_1}T_{e^{ i(\varphi_0+\varphi_1)}z_2}\\
=
e^{ i\im(\overline{z_0}e^{ i\varphi_0}z_1 )}
e^{ i\im(\overline{z_0+e^{ i\varphi_0}z_1}e^{ i(\varphi_0+\varphi_1)}z_2)}Id=e^{i\varepsilon}Id.
\end{multline*}
Asymptotics~\eqref{eq-asymp5} and Eq.~\eqref{ht1} give the desired expression~\eqref{fi2} for  $\Psi_2$.
\end{proof}

\section{Operators on $\mathbb{R}^n$}\label{sec3}
\subsection*{The Euler operator.}
We introduce the Euler operator 
\begin{equation}\label{euler2s}
D=\frac 1{\sqrt{2}}\left(d+d^*+xdx\wedge +(xdx\wedge)^*\right)
\colon \mathcal{S}(\mathbb{R}^n,\Lambda^{ev}(\mathbb{R}^n)\otimes \mathbb{C})\longrightarrow
\mathcal{S}(\mathbb{R}^n,\Lambda^{odd}(\mathbb{R}^n)\otimes \mathbb{C}).
\end{equation}
Here $d$ is the exterior differential,  $xdx\wedge$ is the operator
of exterior multiplication by $xdx=dr^2/2=\sum_j x_jdx_j$, where   $r=|x|$, while $d^*$ and $(xdx\wedge)^*$
stand for the adjoint operators. Its symbol is invertible for $|x|^2+|p|^2\ne0$.\footnote{Indeed, $\sigma(D)(x,p)=2^{-1/2}[(ip+xdx)\wedge+((ip+xdx)\wedge)^*]$.  Hence,
$\sigma(D)^2(x,p)=2^{-1}(|x|^2+|p|^2)Id.$} 
We consider this operator in the Schwartz spaces of complex-valued differential forms. 
Below, we will use the identification $\Lambda (\mathbb{R}^n)\otimes \mathbb{C}\simeq \Lambda(\mathbb{C}^n)$.
According to \cite[Lemma 14]{HKT1}
\begin{equation}\label{di1}
D^2= H+ F, \quad \text{ where }   H=\frac 1 2 \sum_{j=1}^n\left(-\frac{\partial^2}{\partial x_j^2}+x_j^2\right),\quad    F|_{\Lambda^k}= \left(k-\frac n 2\right)Id.
\end{equation}
\subsection*{Heisenberg-Weyl operators.} Given  $z=a-ik\in\mathbb{C}^n$, where $a,k\in \mathbb{R}^n$, we define the operators
$$
T_z u(x)=e^{ikx-iak/2}u(x-a).
$$
These operators  and $e^{i\varepsilon}$ for all $\varepsilon\in\mathbb{R}$ define the so-called Schr\"odinger representation of the Heisenberg group. We note the following product formula
\begin{equation}
\label{eq-comm1}
T_{z_1}T_{z_2}=e^{-i\im( z_1,z_2)/2}T_{z_1+z_2},\quad \text{ where } (z_1,z_2)=z_1\overline{z_2}.
\end{equation}
The Heisenberg-Weyl operators are extended to the space of forms by the trivial action on the differentials, and this extension is denoted by the same symbol. 

\subsection*{Metaplectic operators.} 
Let $g\in U(n)\mapsto R_g\in\mathcal{B}L^2(\mathbb{R}^n)$ be the unitary representation of $U(n)$ by operators in the complex metaplectic group defined in Section \ref{sec-mp}. 
By Theorem~7.13 in~\cite{deG1}
\begin{equation} 
\label{cr12}
R_g T_z R_g^{-1}=  T_{g z}.
\end{equation}

 The metaplectic operators are also extended to the space of forms by the formula:
\begin{equation}
\label{eq-rg2}
\begin{array}{ccc}
 \mathbf{R}_g:  \mathcal{S}(\mathbb{R}^n,\Lambda (\mathbb{C}^n) ) & \longrightarrow & \mathcal{S}(\mathbb{R}^n,\Lambda (\mathbb{C}^n) )\\
 u\otimes \omega & \longmapsto & \mathbf{R}_g(u\otimes\omega)=R_gu\otimes {g^*}^{-1}\omega,
\end{array}
\end{equation}
where ${g^*}^{-1}: \Lambda (\mathbb{C}^n) \to \Lambda (\mathbb{C}^n)  $
stands for the induced action on forms for $g:\mathbb{C}^n\to\mathbb{C}^n.$ 
 
One has the following properties:
\begin{equation}\label{clifford1}
 [D,T_z]=\frac 1{\sqrt 2} T_z  c(z), \text{ where }c(z)=  \overline zdx\wedge+  z  dx \lrcorner ;
\end{equation} 
\begin{equation}\label{eq-equiv2a}
 [D,\mathbf{R}_g]=0, \text{ for all }g\in U(n).
\end{equation} 
Equality~\eqref{clifford1} is straightforward, while \eqref{eq-equiv2a} is \cite[Lemma 4]{SaSc2}.
 
\subsection*{Main results.}
 
Let $\A$ be the operator algebra generated by the operators $T_z$ and 
$\mathbf{R}_g$ for $z\in\mathbb{C}^n, g\in U(n)$ on the graded Hilbert space 
$\H= L^2(\mathbb R^n, \Lambda(\mathbb C^n))$. 
It follows from~\eqref{eq-comm1} and~\eqref{cr12} that an arbitrary element in $\A$ 
can be written as a finite sum
\begin{equation}
\label{eq-genel1}
a=\sum_{z,g} a_{z,g}T_z\mathbf{R}_g:L^2(\mathbb{R}^n,\Lambda(\mathbb{C}^n) )\longrightarrow
L^2(\mathbb{R}^n,\Lambda(\mathbb{C}^n) ),\qquad a_{z,g}\in\mathbb{C}. 
\end{equation}
By $\Psi(\A,\H,D)$ we denote the algebra of all operators  of the form
$$
B=\sum_k D_kT_{z_k}\mathbf{R}_{g_k},
$$
where the sum is finite, $z_k\in\mathbb{C}^n,g_k\in U(n)$ and the $D_k$ are Shubin type pseudodifferential operators (see~Section~\ref{sec-shubin}).  This algebra $\Psi(\A,\H,D)$ might be larger than the one defined by Connes and Moscovici (see Section~\ref{sec1}).

\begin{theorem}\label{th34}
The conditions in Connes--Moscovici's local index theorem $($Theorem~$\ref{th0})$ are satisfied for the graded spectral triple $(\A,\H,D)$. More precisely,  the spectral triple is regular, finitely summable, and has simple dimension spectrum.
\end{theorem}

\begin{proof}
The regularity of the spectral triple follows from the invariance of 
the Shubin pseudodifferential calculus under the affine metaplectic group
generated by all $T_z$ and $\mathbf{R}_g$, $z\in\mathbb{C}^n,g\in U(n)$.   
Moreover, the explicit description of the spectrum of $D^2=H+F$ enables one to prove that $|D|^{-1}$ is $p$-summable whenever $p>2n.$

The proof that  zeta functions for this spectral triple extend to meromorphic functions on $\mathbb{C}$ with simple poles is deferred to Section \ref{sec:zeta}. 
\end{proof}

To state the main result of this paper, we recall the definition of the Berezin integral. 
Given a complex subspace  $L\subset \mathbb{C}^n$, we define the  Berezin integral as a linear functional 
\begin{equation}
\label{eq-bere1}
\int_{L}:\Lambda( L)  \longrightarrow   \mathbb{C}
\end{equation} 
on exterior forms on $L$ considered as a real vector space of dimension $2k$. To define this functional, we choose an orthonormal base $e_1,...,e_k$ in $L$, denote the coordinates with respect to this base by $z_j=p_j+ix_j$ and consider the volume form $dp_1\wedge dx_1\wedge\ldots\wedge dp_k\wedge dx_k\in\Lambda^{2k}(L).$ Then the Berezin integral~\eqref{eq-bere1} is characterized by the properties: $\int_L dp_1\wedge dx_1\wedge\ldots\wedge dp_k\wedge dx_k=1$ and $\int_L\omega=0$ whenever $\deg\omega<2k$. It is easy to show that this definition does not depend on the choice of an orthonormal base. Below  we denote the coordinates in $\mathbb{C}^n$ by $z=p+ix$.

\begin{theorem}\label{th-main3}
The components of the Connes--Moscovici 
cocycle $\Psi=(\Psi_0,\Psi_2,...,\Psi_{2n})$ of the spectral triple in Theorem~\ref{th34} are equal to  
\begin{equation}\label{fmain1}
\Psi_{2k}(a_0,a_1,...,a_{2k})=
\left\{
\begin{array}{l}
\displaystyle 0,\quad \text{if the  mapping $w\mapsto gw+z$  has no fixed points or }k>\dim \mathbb{C}^n_g\vspace{2mm}\\  
\displaystyle\frac{ i^{ -k}}{(2k)!}e^{i\varepsilon}   \prod_{j=1}^m  e^{\frac i 4|(z,e_j)|^2\ctg(\varphi_j/2) } 
\int_{\mathbb{C}^n_g} \sigma(w_1)\wedge\sigma( w_2)\wedge...\wedge \sigma( w_{2k})   \wedge e^{-\omega}
\text{ else,}
\end{array} 
\right.
\end{equation}
where
\begin{itemize}
\item  $a_j=T_{z_j}\mathbf{R}_{g_j}$, $z_j\in\mathbb{C}^n$, $g_j\in U(n)$;
\item $\mathbb{C}^n_g\subset\mathbb{C}^n$ is the fixed point set of $g=g_0 g_1...g_{2k}$; $m=n-\dim_{\mathbb C}  \mathbb{C}^n_g$; 
\item $e^{i\varphi_j}$ for $1\le j\le m$ are the eigenvalues of $g$ which are $\ne 1$, while $e_j\in\mathbb{C}^n$ stand for the corresponding orthonormal system of eigenvectors;
\item $z=w_0+w_1+...w_{2k},$  where $w_j=(g_0g_1...g_{j-1})z_j;$ 
\item $ \varepsilon=\im(\sum_{j\ge k}w_j\overline{w}_k)/2$; 
\item $\sigma(a-ik)=-k dx+a dp\in \Lambda^1(\mathbb{R}^{2n})$;
\item $\omega=\sum_{j=1}^n dx_j\wedge dp_j$  is the symplectic form on $\mathbb{R}^{2n}$;
\item $\int_{\mathbb{C}^n_g}:\Lambda(\mathbb{C}^n_g)\to \mathbb{C}$ stands for the Berezin integral on $\mathbb{C}^n_g\subset\mathbb{C}^n$. 
\end{itemize}
\end{theorem}
\begin{remark}
Note that the affine mapping $w\mapsto gw+z$ is equal to the composition of the affine mappings $w\mapsto g_jw+z_j$
for $j=0,1,...,2k$.
\end{remark}
\begin{remark}
If $n=1$, then \eqref{fmain1} coincides with the 
Connes--Moscovici cocycle in \eqref{fi0} and \eqref{fi2}.  
\end{remark} 

\subsection*{Proof of Theorem~\ref{th-main3}.}
The proof is divided into three steps. First, we compute the heat trace asymptotics for scalar operators. We then use these asymptotics and Getzler's calculus to identify the contribution of the terms with $\alpha=0$ in the Connes--Moscovici formula. Finally we show that the contributions of the terms with $\alpha\ne 0$ are equal to zero.

\subsection*{Step 1. Heat trace asymptotics for scalar operators in $\mathbb{R}^n$.}

Given $g\in U(n)$, we diagonalize it: $g=hg_0h^{-1}$, where $g_0,h\in U(n)$, while $g_0={\rm diag}(e^{i\varphi_1},...,e^{i\varphi_m},1,...,1)$,
$\varphi_1,...,\varphi_m\in (0,2\pi)$, $\varphi_{m+1}=...=\varphi_n=0$. 
Then we have
\begin{equation}\label{eq6}
\tr (T_zR_ge^{-t H})=\tr (T_zR_h R_{g_0}R_h^{-1}e^{-t  H})=
\tr (R_h^{-1}T_z R_h R_{g_0}e^{-t  H})= \tr (T_{h^{-1}z} R_{g_0}e^{-t  H}).
\end{equation}
In the second equality we used the fact that $R_h$ commutes with $  H$, while in the last  we used \eqref{cr12}. 
Since $g_0$ is diagonal, the trace in \eqref{eq6} is the product of the traces of $n$ operators on $\mathbb{R}$:
$$
\tr (T_zR_ge^{-t  H})=\prod_{j=1}^n \tr ( T_{(h^{-1}z)_j} R_{\varphi_j}e^{-t H_j}),\quad
H_j=\frac{1}{2}(x_j^2-\partial^2_{x_j}),
$$
where $(h^{-1}z)_j$ is the $j$-th component of $h^{-1}z$.
We now apply the one-dimensional heat expansion of  Proposition~\ref{prop1} and obtain the following asymptotics.
\begin{proposition}\label{prop4}
\begin{equation}\label{ht6}
 \tr (T_zR_ge^{-t H})\sim
 \left\{
 \begin{array}{l}
  \displaystyle  O(t^{+\infty}) \quad \text{ if the affine mapping $w\to gw+z$ has no fixed points}; \vspace{2mm}\\
    \displaystyle t^{-(n-m)} \prod_{j=1}^m \frac {e^{ \frac i 4|(z,e_j)|^2
    \ctg(\varphi_j/2) }}{ 1-e^{-i \varphi_j} } 
\quad  \text{ otherwise}.\\
 \end{array}
\right.
\end{equation}
As before, $\{e_j\}_{j=1}^m$ is an orthonormal system of eigenvectors of $g$ with eigenvalues $e^{i\varphi_j}$; $e^{i\varphi_j}\ne 1$ for $j=1,\ldots, m$. 
Note that $(z,e_j)=(h^{-1}z)_j$ and that the condition that the fixed point set of $w\mapsto gw+z$ is nontrivial is equivalent to   $( z)_j=(z,e_j)= 0 \text{ for all }j>m$ (equivalently, $z$ is orthogonal to the fixed point set of $g$).

Also, if $B$ is a differential operator of order $\ord B$, then one has
\begin{equation}\label{ht6b}
 \tr (BT_zR_ge^{-t H})\sim
 \left\{
 \begin{array}{l}
  \displaystyle  O(t^{+\infty}) \quad \text{ if the   mapping $w\to gw+z$ has no fixed points}; \vspace{2mm}\\
    \displaystyle O(t^{-(n-m+\ord B/2)})  
\quad  \text{ otherwise}.\\
 \end{array}
\right.
\end{equation}
\end{proposition}

\subsection*{Step 2. Computation of the Connes--Moscovici cocycle for $\alpha=0$.}
Given $a_j=T_{z_j}\mathbf{R}_{g_j}$, we have\footnote{Here we used the property ${g^{-1}}^*c(z)=c(gz){g^{-1}}^*$,
which is easy to prove using the property $g^*(v\lrcorner \omega) =(g_* v)\lrcorner (g^*\omega)$ for a vector $v$ and a differential form $\omega$.} 
\begin{multline}\label{eq-la1}
a_0[D,a_1]...[D,a_{2k}]e^{-tD^2}=
2^{-k}T_{z_0}\mathbf{R}_{g_0}c(z_1)T_{z_1}\mathbf{R}_{g_1}... c(z_{2k})T_{z_{2k}}\mathbf{R}_{g_{2k}}e^{-t(  H+ F)} \\
=2^{-k}\Bigl(T_{z_0}R_{g_0}T_{z_1}R_{g_1}...  T_{z_{2k}}R_{g_{2k}}e^{-t  H}\Bigr)\otimes
\Bigl( {g^{-1}_0}^*c(z_1) {g^{-1}_1}^*c(z_2){g^{-1}_2}^*...c(z_{2k}){g^{-1}_{2k}}^* e^{-t   F}\Bigr) \\
=2^{-k}\Bigl(e^{i\varepsilon}T_zR_g e^{-t H}\Bigr)\otimes\Bigl(c(w_1) c(w_2)...c(w_{2k}){g^{-1}}^* e^{-t   F}\Bigr),
\end{multline}
where  
$$
z=w_0+w_1+w_2+...,\quad w_j=(g_0g_1...g_{j-1})z_j,  
$$
while $e^{i\varepsilon}\text{Id}=T_{w_0}T_{w_1}T_{w_2}...T_{w_{2k}}T_z^{-1}$. 
Hence, we obtain for the supertrace
\begin{equation} \label{susy0}
\tr_s(a_0[D,a_1]...[D,a_{2k}]e^{-tD^2})=2^{-k}e^{i\varepsilon}\tr\Bigl(T_zR_g e^{-t  H}\Bigr)\tr_s \Bigl(c(w_1) c(w_2)...c(w_{2k}){g^{-1}}^* e^{-t   F}\Bigr).
\end{equation}
The heat trace  $\tr\Bigl(T_zR_g e^{-t  H}\Bigr)$ here is computed using \eqref{ht6}.
Note that \eqref{ht6} is exponentially small if the fixed point set of the affine mapping $w\to gw+z$ is empty.
 It remains to compute the expansion of the supertrace in \eqref{susy0}.

We have $w_j\in \mathbb{C}^n\simeq \mathbb{R}^{2n}$ with the base $e_1,e_2,...,e_{2n-1},e_{2n}$:
$$
e_1=(i,0,0,...,0),\; e_2=(1,0,0,...,0),\; ..., e_{2n-1}=(0,0,...,0,i),\; e_{2n}=(0,0,...,0,1)
$$
and let $Cl(2n)$ be the real Clifford algebra generated by these vectors with the relations
$$
e_j^2=1 \text{ for all }j;\quad e_je_k+e_ke_j=0 \text{ for all }k\ne j .
$$ 
The mapping $z\mapsto c(z)\in\End\Lambda(\mathbb{C}^n)$ defined in \eqref{clifford1} enjoys the property
$$
c(z_1)c(z_2)+c(z_2)c(z_1)=2\re (z_1\cdot \overline z_2).
$$
Thus,  $c(e_j)^2=1$, and $c(e_j)c(e_k)+c(e_k)c(e_j)=0$, $k\not=j$. Hence, this mapping uniquely extends to a homomorphism of  algebras
denoted by $c:Cl(2n)\to\End\Lambda(\mathbb{C}^n) $. 

The symbol mapping $\sigma$ is given by 
$$
 \begin{array}{ccc}
\sigma: Cl(2n) & \lra &  \Lambda(\mathbb{R}^{2n})\\
  e_{j_1}e_{j_2}...e_{j_k} & \longmapsto & \sigma(e_{j_1})\wedge \sigma(e_{j_2})\wedge ...\wedge \sigma(e_{j_k}),
\end{array} 
$$
where  $j_1,...,j_k$ are all different, and $\sigma(e_{2j-1})=dx_j$, $\sigma(e_{2j})=dp_j$. 

In the sequel we shall need the Berezin lemma:
$$
\tr_s(c(a))=(-2i)^n\int_{\mathbb{C}^{2n}} \sigma(a),\quad \forall a\in Cl(2n).
$$
To check that the constant here is chosen correctly, we consider the special case $n=1$. Then $Cl(2)$ is spanned by $1,e_1,e_2,e_1e_2$. Both sides of the formula are nontrivial only for $a=e_1e_2$ and we have
$$
c(e_1)c(e_2)= (-idx\wedge+idx\lrcorner) (dx\wedge+dx\lrcorner)=-i(dx\wedge dx\lrcorner-dx\lrcorner dx\wedge).
$$
Hence, we get $\tr_s(c(e_1)c(e_2))=i+i=2i$. On the other hand, we have $\sigma(e_1e_2)=dx\wedge dp$ and $\int_\mathbb{C}\sigma(e_1e_2)=\frac{dx\wedge dp}{dp\wedge dx}=-1$.

\begin{proposition}\label{p5} Given $k\le n-m$, we have
\begin{multline}\label{fermi2}
 \tr_s \Bigl(c(w_1) c(w_2)...c(w_{2k}){g^{-1}}^* e^{-t  F}\Bigr) \\
\sim
t^{n-m-k}2^ki^{-k}
\prod_{j=1}^m  \left( 1-e^{-i\varphi_j}\right)  \int_{\mathbb{C}^n_g} 
\sigma(w_1)\wedge\sigma( w_2)\wedge...\wedge \sigma( w_{2k})   \wedge e^{-\omega} ,
\end{multline}
where $\omega=dx_1\wedge dp_1+\ldots + dx_n\wedge dp_n$ is the symplectic form on $\mathbb{C}^n$
with coordinate $z=p+ix$.
Also if $k>n-m$, then the left hand side in \eqref{fermi2} is $O(1)$.
\end{proposition}

\begin{proof}
If $k>n-m$, then the statement in this proposition is true, 
since the expression is smooth up to $t=0$.   Let us now obtain the asymptotics for $k\le n-m$.
Since both sides of the formula are invariant under changes of coordinates, we choose coordinates, in which $g$ is a diagonal matrix $\diag(e^{i\varphi_1},...,e^{i\varphi_m},1,...,1)$.

\begin{lemma}\label{lem6}
The operators $ F,e^{-t F}, {g^{-1}}^*\in \End \Lambda(\mathbb{C}^n)$ are expressed in terms of the Clifford multiplication as
$$
  F = \frac i 2 \sum_{j=1}^n c(e_{2j-1}e_{2j}),\quad e^{-t  F}=\prod_{j=1}^n\left(\ch \frac t 2 -ic(e_{2j-1}e_{2j})\sh \frac t 2 \right)
$$
\begin{equation}\label{g-cliff1}
{g^{-1}}^*=\prod_{j=1}^m \left(\cos\frac {\varphi_j}2+\sin\frac{\varphi_j}2 c(e_{2j-1}e_{2j})\right)e^{-i\varphi_j/2}.
\end{equation}
\end{lemma}
\begin{proof}
The proof is straightforward.
\end{proof}

We now use Lemma~\ref{lem6}  and the Berezin lemma  to obtain
\begin{multline}\label{fermi4}
 \tr_s \Bigl(c(w_1) c(w_2)...c(w_{2k}){g^{-1}}^* e^{-t  F}\Bigr) \\
= \tr_s \Bigl(c(w_1) c(w_2)...c(w_{2k})  \prod_{j=1}^m \left(\cos\frac {\varphi_j}2+ \sin\frac{\varphi_j}2 c(e_{2j-1}e_{2j})\right)e^{-i\varphi_j/2}   
 \prod_{j=1}^n\left(\ch \frac t 2 -ic(e_{2j-1}e_{2j})\sh \frac t 2 \right)\\
=
(-2i)^n\int_{\mathbb{C}^n} \sigma\left(
w_1 w_2... w_{2k}\prod_{j=1}^m \left(\cos\frac {\varphi_j}2+ \sin\frac{\varphi_j}2  e_{2j-1}e_{2j} \right)e^{-i\varphi_j/2}   
 \prod_{j=1}^n\left(\ch \frac t 2 -i e_{2j-1}e_{2j} \sh \frac t 2 \right) 
\right).
\end{multline}
Since $2k\le 2n-2m$ by assumption, the main term of the expansion of \eqref{fermi4} is of order $t^{n-m-k}$ and equal to 
\begin{multline}\label{fermi5}
(-2i)^n\int_{\mathbb{C}^n} \sigma\left(
w_1 w_2... w_{2k}\prod_{j=1}^m \left(\cos\frac {\varphi_j}2+ \sin\frac{\varphi_j}2  e_{2j-1}e_{2j} \right)e^{-i\varphi_j/2}   
 \prod_{j=1}^n\left(\ch \frac t 2 -i e_{2j-1}e_{2j} \sh \frac t 2 \right) 
\right) \\
\sim(-2i)^n (-it/2)^{n-m-k}
\int_{\mathbb{C}^n} \sigma\left(
w_1 w_2... w_{2k}\prod_{j=1}^m \left(   \sin\frac{\varphi_j}2  e_{2j-1}e_{2j} \right)e^{-i\varphi_j/2}   
\sum_J \prod_{j\in J} e_{2j-1}e_{2j}  \right) =\\
=(-2i)^n  (-it/2)^{n-m-k}
\prod_{j=1}^m \sin\frac{\varphi_j}2 e^{-i\varphi_j/2}   \times \\
\times \int_{\mathbb{C}^n} 
\sigma(w_1)\wedge\sigma( w_2)\wedge...\wedge \sigma( w_{2k}) \wedge  dx_1\wedge dp_1\wedge...\wedge dx_m\wedge dp_m \wedge
\sum_J \prod_{j\in J} dx_j\wedge dp_j \\
=(-2i)^n  (-it/2)^{n-m-k}(-1)^m\prod_{j=1}^m  \frac{1-e^{-i\varphi_j} }{2i}    \int_{\mathbb{C}^n_g} 
\sigma(w_1)\wedge\sigma( w_2)\wedge...\wedge \sigma( w_{2k})   \wedge e^\omega\\
=t^{n-m-k}2^ki^{-k}
\prod_{j=1}^m  \left( 1-e^{-i\varphi_j}\right)  \int_{\mathbb{C}^n_g} 
\sigma(w_1)\wedge\sigma( w_2)\wedge...\wedge \sigma( w_{2k})   \wedge e^{-\omega}.
\end{multline}
Here the summations $\sum_J$ are over all subsets $J\subset\{m+1,...,n\}$ of $n-m-k$ different numbers.

The proof of Proposition~\ref{p5} is now complete. 
\end{proof}

Now we substitute the asymptotics \eqref{ht6} and \eqref{fermi2} into \eqref{susy0} and obtain the desired expression \eqref{fmain1} for the Connes--Moscovici 
cocycle.

\subsection*{Step 3. Computation of the Connes--Moscovici cocycle for $\alpha\ne 0$.}

We claim that for $\alpha\ne 0$ the contribution to the Connes--Moscovici 
cocycle is equal to zero. Indeed, similar to \eqref{eq-la1} we get
\begin{multline}\label{eq-bb1}
\tr_s(a_0[D,a_1]^{[\alpha_1]}\ldots [D,a_{2k}]^{[\alpha_{2k}]}e^{-tD^2})\\
=\tr_s(T_{w_0}[D,T_{w_1}]^{[\alpha_1]}\ldots [D,T_{w_{2k}}]^{[\alpha_{2k}]}\mathbf{R}_g e^{-t(H+F)})\\
 =\tr_s(T_{w_0}[D,T_{w_1}]^{[\alpha_1]}\ldots [D,T_{w_{2k}}]^{[\alpha_{2k}]}({R}_g\otimes {g^*}^{-1})(  e^{-tH}\otimes e^{-tF})).
\end{multline}
To study the asymptotics of this supertrace, we recall the following definition of Getzler order.

\begin{definition}
Given an operator
\begin{equation}
\label{bb-1}
B=\sum_k \left(B_k T_{z_k}\right)\otimes c(a_k): \mathcal{S}(\mathbb{R}^n,\Lambda(\mathbb{C}^n))\longrightarrow \mathcal{S}(\mathbb{R}^n,\Lambda(\mathbb{C}^n)), 
\end{equation}
where $B_k$ are scalar Shubin differential operators, $z_k\in \mathbb{C}^n$, $a_k\in Cl(2n)$,
its {\em Getzler order}  is equal to
$$
 \ord B=\max_k(\ord B_k+\deg a_k).
$$
Thus, in the order we count the order in $x,\partial/\partial x$ and the Clifford filtration.\footnote{Recall that $Cl(2n)$ is a filtered algebra and we define $\ord a$ for $a\in Cl(2n)$ to be equal to the Clifford filtration $d$, where
$Cl_d(2n)\subset Cl(2n)$ is the subspace of elements spanned by the products $v_1\cdots v_d\in Cl(2n)$, where $v_j\in \mathbb{C}^{n}\subset Cl(n)$.}
\end{definition}

The Getzler orders of the operators in \eqref{eq-bb1} are computed in the following lemma.

\begin{lemma}\label{lem-get1}
One has $\ord {g^*}^{-1}= 2m$ and  $\ord [D,T_w]^{[\gamma]}\le 1+\gamma$.
\end{lemma}

\begin{proof}
The first equality follows from \eqref{g-cliff1}. The second estimate is proved by induction. Indeed, if $\gamma=0$, then \eqref{clifford1} shows that $\ord [D,T_w]\le 1$. Let us now show that $\ord[D^2,B]\le \ord B+1$ for all $B$ as in \eqref{bb-1}. We have:
\begin{equation}\label{bbc-1}
[D^2,B]=[H+F,B]=[H,B]+[F,B] 
=\sum_k\left( [H,B_kT_{z_k}]\otimes  c(a_k)+ B_kT_{z_k}\otimes [F,c(a_k)]\right).
\end{equation}
It follows from the properties of Shubin operators that $\ord [H,B_kT_{z_k}]\le \ord B_k+1$, and from the properties of the Clifford multiplication that $\ord[F,c(a_k)]\le \ord a_k+1$. These estimates and \eqref{bbc-1} imply the desired estimate
$$
\ord [D^2,B]\le \max_k(\ord B_k+\ord a_k+1)=\ord B+1.
$$
The proof of Lemma~\ref{lem-get1} is now complete.
\end{proof}

\begin{lemma}\label{lem-get2}
Given an operator $B$ as in \eqref{bb-1}, we have
\begin{equation}\label{eq-esta2}
\tr_s(B\mathbf{R}_ge^{-tD^2})=
\left\{ 
\begin{array}{l}
O(t^{+\infty}) \text{ if the fixed point set of $w\mapsto gw+z$ is trivial}\vspace{2mm}\\
O(t^{-\ord B/2}) \text{ otherwise}.
\end{array}
\right.
\end{equation} 
\end{lemma}
\begin{proof}
We have
\begin{multline}
\label{comba3}
\tr_s(B\mathbf{R}_ge^{-tD^2})=\sum_k \tr_s\bigl((B_kT_{z_k}\otimes c(a_k))( {R}_g\otimes  {g^*}^{-1})(e^{-tH}\otimes e^{-tF})\bigr)\\
= \sum_k \tr (B_kT_{z_k}{R}_ge^{-tH}) \tr_s( c(a_k)  {g^*}^{-1}e^{-tF}).
\end{multline}
On the one hand, the trace of scalar operators is estimated by Proposition~\ref{prop4}:
\begin{equation}
\label{eq-bos1}
\tr (B_kT_{z_k}{R}_ge^{-tH})=
\left\{ 
\begin{array}{l}
O(t^{+\infty}) \text{ if the fixed point set of $w\mapsto gw+z$ is trivial}\vspace{2mm}\\
O(t^{-\dim\mathbb{C}^n_g-\ord B_k/2}) \text{ otherwise.}
\end{array}
\right.
\end{equation}
On the other hand,  the supertrace in \eqref{comba3} is computed by Proposition~\ref{p5}:
\begin{equation}
\label{eq-ferm1}
\tr_s( c(a_k)  {g^*}^{-1}e^{-tF})=
\left\{ 
\begin{array}{l}
O(1) \text{ if $\ord a_k$ is odd or $\dim\mathbb{C}^n_g-\ord a_k/2<0$}\vspace{2mm}\\
O(t^{\dim\mathbb{C}^n_g-\ord a_k/2}) \text{ otherwise.}
\end{array}
\right.
\end{equation}
Estimating the traces in \eqref{comba3} using \eqref{eq-bos1} and \eqref{eq-ferm1}, we obtain the desired estimate
\eqref{eq-esta2}.
\end{proof}

Now, we see from Lemma~\ref{lem-get1} that
$$
 \ord [D,a_1]^{[\alpha_1]},..., [D,a_{2k}]^{[\alpha_{2k}]}\le 2k+|\alpha|.
$$
Thus, Lemma~\ref{lem-get2} implies that
$$
\tr_s(a_0[D,a_1]^{[\alpha_1]}... [D,a_{2k}]^{[\alpha_{2k}]}e^{-tD^2})=
O(t^{-(2k+|\alpha|)/2}).
$$
Hence
$$
t^{k+|\alpha|}\tr_s(a_0[D,a_1]^{[\alpha_1]},..., [D,a_{2k}]^{[\alpha_{2k}]}e^{-tD^2})=
O(t^{ |\alpha|/2}) 
$$
and the constant term in the asymptotic expansion is equal to zero. This implies the desired statement that
the terms in the Connes--Moscovici 
cocycle for $\alpha\ne 0$ are equal to zero. 

The proof of Theorem~\ref{th-main3} is now complete.

\section{Cyclic Cocycles}\label{sec4}

In this section, we show that  each component of the periodic cyclic cocycle in Theorem~\ref{th-main3} is actually a cyclic cocycle. Moreover, each of these cocycles is a sum of cyclic cocycles localized  at the conjugacy classes in the semidirect product of  $\mathbb{C}^n$ and the unitary group $U(n)$ which we denote by $\mathbb{C}^n\rtimes U(n)$. Here we use the approach due to Connes to define cyclic cocycles as characters of cycles, see~\cite{Con1}.

\subsection*{Noncommutative differential forms.}

We consider $\A$ as a subalgebra of  the 
differential graded algebra $\Omega^*\subset 
\mathcal{B}L^2(\mathbb{R}^n,\Lambda(\mathbb{R}^{2n}))$ consisting of all operators $a$
that are  finite sums 
\begin{equation}
\label{eq-diff1} 
a=\sum_{k}  u_k T_{z_k}\mathbf{R}_{g_k},\quad z_k\in\mathbb{C}^n, g_k\in U(n), u_k\in \Lambda(\mathbb{R}^{2n}).
\end{equation}
This algebra is graded by the degree of forms.  
We define the operator $d:\Omega^*\to \Omega^{*+1}$ by
\begin{equation}
\label{eq-diff2}
 d (u T_z\mathbf{R}_g  )  =(-1)^{\deg u}u\sigma(z)T_z \mathbf{R}_g,\quad\text{ where }\sigma(z)=\im z dx+\re z dp.
\end{equation}
\begin{lemma}
\label{lem11}
The operator $d$ is  a graded differentiation on $\Omega^*$. More precisely, the following equalities hold:
$$
d^2a=0,\quad d(a_1 a_2)=(da_1) a_2+(-1)^{\deg a_1} a_1 da_2,\quad\text{ for all }a,a_1,a_2\in\Omega^*.
$$
\end{lemma}

\begin{proof}  
The first equality is seen as follows
$$
d (d(u T_z\mathbf{R}_g  )   )=d((-1)^{\deg u} u\sigma(z) T_z \mathbf{R}_g    )=-u  \sigma(z)\sigma(z)T_z\mathbf{R}_g     =
u \cdot 0\cdot T_z \mathbf{R}_g     =0.
$$
Before proving the second, we first show that   
\begin{equation}
\label{eq-equiv}
{g^*}^{-1}(\sigma(z))=\sigma(gz) \quad \text{ for all $g\in {U}(n)$ and $z\in\mathbb{C}^n$}.
\end{equation}
In fact, given $g=B+iA\in  {U}(n)$, where $A$ and $B$ are real matrices, we have on the one hand 
\begin{multline}
\label{eq-left3}
\sigma(gz)=\re(gz d\overline{(p+ix)})=\sum_{kl}\re\left((B_{kl}+iA_{kl})(\re z_l+i\im z_l)(dp_k-idx_k)\right)\\
=\sum_{kl}\left(
B_{kl}(\re z_l dp_k+\im z_l dx_k)+A_{kl}(-\im z_l dp_k+\re z_l dx_k)\right).
\end{multline}
On the other hand, 
$g^{-1}=B^t-iA^t$ with the transposed matrices $A^t$ and $B^t$, and
$$
g^{-1}\begin{pmatrix}
 x \\
 p
\end{pmatrix}=
\begin{pmatrix}
 B^t & -A^t\\
 A^t & B^t
\end{pmatrix}
\begin{pmatrix}
 x \\
 p
\end{pmatrix}.
$$
Now  \eqref{eq-equiv} follows from the fact that 
\begin{multline}
\label{eq-right3}
{g^*}^{-1}\sigma(z)= \sum_l (\im z_l ){g^*}^{-1}(dx_l)+\sum_l (\re z_l) {g^*}^{-1}(dp_l)\\
=\sum_{kl} (\im z_l)(B_{kl}dx_k-A_{kl}dp_k)+\sum_{kl}(\re z_l )(A_{kl}dx_k+B_{kl}dp_k)
\end{multline}
which coincides with \eqref{eq-left3}.

As for the second statement: Given $a_1=u_1T_{z_1}\mathbf{R}_{g_1}  $ and  $a_2= u_2T_{z_2}\mathbf{R}_{g_2} $, we find  that 
\begin{eqnarray*}
\lefteqn{d(a_1a_2)= d(( u_1T_{z_1}\mathbf{R}_{g_1} )( u_2 T_{z_2}\mathbf{R}_{g_2} ))=d(u_1{g_1^*}^{-1}u_2T_{z_1}\mathbf{R}_{g_1}T_{z_2}\mathbf{R}_{g_2}   )}\\
&=&e^{-i\im(z_1,g_1z_2)/2}d(u_1{g_1^*}^{-1}u_2T_{z_1+g_1z_2}\mathbf{R}_{g_1g_2}   )\\
&=&(-1)^{\deg u_1+\deg u_2}e^{-i\im(z_1,g_1z_2)/2}u_1{g_1^*}^{-1}u_2\sigma(z_1+g_1z_2)T_{z_1+g_1z_2} \mathbf{R}_{g_1g_2}   \\
&=&(-1)^{\deg u_1 } u_1\sigma(z_1 ){g_1^*}^{-1}u_2T_{z_1}T_{g_1z_2} \mathbf{R}_{g_1g_2}+
(-1)^{\deg u_1+\deg u_2} u_1{g_1^*}^{-1}u_2\sigma(g_1z_2)T_{z_1}T_{g_1z_2} \mathbf{R}_{g_1g_2}\\
&=&(-1)^{\deg u_1 } u_1\sigma(z_1 )T_{z_1}\mathbf{R}_{g_1 } u_2T_{ z_2} \mathbf{R}_{ g_2}+
(-1)^{\deg u_1 } u_1T_{z_1}\mathbf{R}_{g_1 }(-1)^{\deg u_2 } u_2\sigma( z_2)T_{ z_2} \mathbf{R}_{ g_2}\\
&=&(da_1) a_2+(-1)^{\deg a_1} a_1 da_2.
\end{eqnarray*}
\end{proof}

\subsection*{Localized traces.}

Let us fix a pair $(z_0,g_0)\in \mathbb{C}^n\times U(n)$ such that the fixed point set of the affine mapping
$$
\mathbb{C}^n\longrightarrow \mathbb{C}^n,\quad w\longmapsto g_0w+z_0
$$
is not empty. Then we define the functional 
$$
\tau_{z_0,g_0}: \Omega^* \longrightarrow \mathbb{C}
$$ 
$$
\tau_{z_0,g_0}\left(\sum_{z,g} u_{z,g}T_z\mathbf{R}_g  \right) =
\sum_{(z,g)\in \langle (z_0,g_0)\rangle} 
 \prod_{j=1}^m  e^{\frac i 4|(z,e_j(g))|^2\ctg\varphi_j(g)/2 }
\int_{\mathbb{C}^n_g} 
u_{z,g}   \wedge e^{-\omega}.
$$
Here $\langle (z_0,g_0)\rangle\subset \mathbb{C}^n\rtimes U(n)$ stands for the conjugacy class of $(z_0,g_0)$, $m=n-\dim \mathbb{C}^n_g$, $e_j(g)$ stand for the eigenvectors of $g$ with eigenvalues $e^{i\varphi_j(g)}\ne 1$.

\begin{lemma}
\label{lem12}
The functional $\tau_{z_0,g_0}$ is a closed graded trace on the differential graded algebra $\Omega^*$. More precisely, one has
$$
\tau_{z_0,g_0} (da)=0\qquad\text{ for all }a\in\Omega^*,
$$
$$
\tau_{z_0,g_0}(a_1a_2)=(-1)^{\deg a_1\deg a_2}\tau_{z_0,g_0}(a_2a_1)\quad 
\text{ for all }a_1,a_2\in\Omega^*.
$$
\end{lemma} 
\begin{proof}
1. Given $a=uT_z\mathbf{R}_g$, we have  
\begin{equation*}
  \tau_{z_0,g_0} (da)=(-1)^{\deg u}\tau_{z_0,g_0} (u\sigma(z)T_z \mathbf{R}_g)={\rm Const}\cdot\int_{\mathbb{C}^n_g}
  u \sigma(z)e^{-\omega}=0,
\end{equation*}
where we used the assumption that the fixed point set of $w\mapsto gw+z$ is nonempty, which is equivalent to  $\sigma(z)|_{\mathbb{C}^n_g}=0.$ Indeed, if we choose the basis in which $g$ is diagonal, then $\mathbb{C}^n_g=\{(0,...,w_{m+1},...,w_n)\}$. Hence, the fixed point set of $w\mapsto gw+z$ is nonempty if and only if $z_j=0$ whenever  $j>m$.
Clearly, this condition is equivalent to $\sigma(z)|_{\mathbb{C}^n_g}=0.$

2. Given $a_j=u_jT_{z_j}\mathbf{R}_{g_j}$, $j=1,2$, we denote by 
$\gamma\subset \mathbb{C}^n\rtimes U(n)$ the conjugacy class of $(z_2,g_2)(z_1,g_1)$
which coincides with that of $(z_1,g_1)(z_2,g_2)$. 
Then we have
\begin{multline}
\label{eq-long1}
\tau_\gamma(a_1a_2)=\tau_\gamma(u_1T_{z_1}\mathbf{R}_{g_1}u_2T_{z_2}\mathbf{R}_{g_2}) =
e^{-i\im(z_1,g_1z_2)/2}\tau_\gamma(u_1{g_1^*}^{-1}u_2T_{z_1+g_1z_2}\mathbf{R}_{g_1g_2}) \\ 
=e^{-i\im(z_1,g_1z_2)/2}\prod_{j=1}^m  e^{\frac i 4|(z_1+g_1z_2,e_j )|^2\ctg(\varphi_j/2) }
\int_{\mathbb{C}^n_{g_1g_2}} 
u_1{g_1^*}^{-1} u_2    e^{-\omega},
\end{multline}
where the $e_j$ are the eigenvectors of $g_1g_2$ with eigenvalues $e^{i\varphi_j}\ne 1$.
Similarly, we get
\begin{equation}
\label{eq-long2} 
\tau_\gamma(a_2a_1)= e^{-i\im(z_2,g_2z_1)/2}\prod_{j=1}^m  e^{\frac i 4|(z_2+g_2z_1,f_j )|^2\ctg(\varphi_j/2) } 
\int_{\mathbb{C}^n_{g_2g_1}} 
u_2{g_2^*}^{-1} u_1    e^{-\omega},
\end{equation}
where the $f_j$ are the eigenvectors of $g_2g_1$ with eigenvalues $e^{i\varphi_j}\ne 1$. 

To compare~\eqref{eq-long1} with~\eqref{eq-long2}, we first compare the integrals. We claim that 
\begin{equation}
\label{eq-int3}
 \int_{\mathbb{C}^n_{g_1g_2}} 
u_1{g_1^*}^{-1} u_2    e^{-\omega}=(-1)^{\deg u_1\deg u_2}\int_{\mathbb{C}^n_{g_2g_1}} 
u_2{g_2^*}^{-1} u_1    e^{-\omega}.
\end{equation}
Indeed, since $g_1^{-1}$ defines an isomorphism $\mathbb{C}^n_{g_1g_2}\to\mathbb{C}^n_{g_2g_1}$, we have
\begin{multline}
\int_{\mathbb{C}^n_{g_2g_1}} 
u_2{g_2^*}^{-1} u_1    e^{-\omega}=\int_{\mathbb{C}^n_{g_1g_2}} ({g_1^*}^{-1}u_2) {g_1^*}^{-1}{g_2^*}^{-1} u_1    e^{-\omega}= \\
=\int_{\mathbb{C}^n_{g_1g_2}}  ({g_1^*}^{-1}u_2)   u_1    e^{-\omega}=(-1)^{\deg u_1\deg u_2}\int_{\mathbb{C}^n_{g_1g_2}} 
u_1{g_1^*}^{-1} u_2    e^{-\omega},
\end{multline}
where we used the fact that ${g_1^*}^{-1}{g_2^*}^{-1}={(g_1g_2)^*}^{-1}=1$ on $\mathbb{C}^n_{g_1g_2}$.

To compare the exponential functions in~\eqref{eq-long1} and~\eqref{eq-long2}, we set $f_j=g_1^{-1}e_j$ and 
claim that
\begin{equation}
\label{eq-long3}
 \frac{\im(z_2,g_2z_1)-\im(z_1,g_1z_2)}2+\frac{1}{4}\sum_{j=1}^m
 \left(
   |(z_1+g_1z_2,e_j )|^2-|(z_2+g_2z_1,f_j )|^2
 \right)
 \ctg\varphi_j/2=0 .
\end{equation}
To prove~\eqref{eq-long3}, we decompose $z_1+g_1z_2$ and $z_1$ as 
$$
z_1=\sum_j d_j e_j+z_{10}, \quad d_j=(z_1,e_j),\quad z_{10}\in \mathbb{C}^n_{g_1g_2},
$$
$$
z_1+g_1z_2=\sum_j c_je_j,\quad c_j=(z_1+g_1z_2,e_j).
$$
Note that $z_1+g_1z_2$ has no component in $\mathbb{C}^n_{g_1g_2}$ since by our assumption the affine mapping $w\mapsto g_1g_2w+z_1+g_1z_2$ has nontrivial fixed point set. Let us now compute the left hand side in~\eqref{eq-long3}.
We have
\begin{multline}
|(z_2+g_2z_1,f_j )|^2=|(z_2+g_2z_1,g_1^{-1}e_j )|^2=|(g_1z_2+g_1g_2z_1,e_j)|^2=|(g_1z_2+ z_1+(g_1g_2-1)z_1,e_j)|^2\\
=|c_j+(e^{i\varphi_j}-1)d_j|^2=|c_j|^2+|d_j|^22(1-\cos\varphi_j)+2\re({c_j}\overline{d_j}(e^{-i\varphi_j}-1)).
\end{multline}
Hence,
\begin{multline}
\im(z_2,g_2z_1)-\im(z_1,g_1z_2)=\im (g_1z_2,g_1g_2z_1)-\im(z_1,g_1z_2)\\
=\im\left(\sum_j(c_j-d_j)e_j-z_{10},z_{10}+\sum_je^{i\varphi_j}d_je_j\right)-
\im(\sum_j d_je_j+z_{10}, \sum_j(c_j-d_j)e_j-z_{10})\\
=\im \sum_j(c_j-d_j)\overline{d_j}e^{-i\varphi_j}-\im\sum_j d_j(\overline{c_j}-\overline{d_j})
=  \sum_j \left(
   \im (c_j \overline{d_j} (e^{-i\varphi_j}+1)) +|d_j|^2\sin\varphi_j 
  \right),
\end{multline}
Thus, the left hand side in~\eqref{eq-long3} is equal to
\begin{multline}
\frac12  \sum_j \left(
   \im (c_j \overline{d_j} (e^{-i\varphi_j}+1)) +|d_j|^2\sin\varphi_j   \right)\\+
  \frac 1 4 \sum_j \left(
  |c_j|^2-|c_j|^2-|d_j|^22(1-\cos\varphi_j)-2\re(c_j\overline{d_j}(e^{-i\varphi_j}-1))\right)\ctg(\varphi_j /2)\\ 
  =\frac 12 \sum_j |d_j|^2(\sin\varphi_j-(1-\cos\varphi_j)\ctg(\varphi_j /2))+
  \frac{1}{2}\sum_j\left[\im (c_j \overline{d_j} (e^{-i\varphi_j}+1))-\re(c_j\overline{d_j}(e^{-i\varphi_j}+1)(-i))\right]\\
  =0+ \frac{1}{2}\sum_j\left[\im (c_j \overline{d_j} (e^{-i\varphi_j}+1))-\im(c_j\overline{d_j}(e^{-i\varphi_j}+1)) \right]=0.
\end{multline}
Here we used the identities 
$$
\sin\varphi_j-(1-\cos\varphi_j)\ctg(\varphi_j /2)=0, \qquad (e^{-i\varphi_j}-1) \ctg(\varphi_j /2)=-i (e^{-i\varphi_j}+1).
$$
Equations \eqref{eq-int3} and \eqref{eq-long3} show that the left hand sides in \eqref{eq-long1} and \eqref{eq-long2}
are related as follows:
$$
\tau_\gamma(a_1a_2)=(-1)^{\deg a_1\deg a_2}\tau_\gamma(a_2a_1).
$$
This equality is precisely the graded trace property. 
The proof of Lemma~\ref{lem12} is now complete.
\end{proof}

\subsection*{Cyclic cocycles.}
Lemmas~\ref{lem11} and~\ref{lem12} imply that $(\Omega^*,\tau_{z_0,g_0})$ is a cycle in the sense of Connes~\cite{Con1}. In a standard way, we define the character of this cycle as the following cyclic cocycle:
\begin{equation}
\label{eq-char1}
 \Phi_{k;z_0,g_0}(a_0,a_1,...,a_k)=\tau_{z_0,g_0} (a_0da_1...da_k),\quad a_j\in \A.
\end{equation}

Theorem~\ref{th-main3} together with Lemmas~\ref{lem11} and~\ref{lem12} implies the following corollary.
\begin{corollary}\label{cor1}
Each component  $\Psi_{2k}$ of the Connes--Moscovici periodic cyclic cocycle~\eqref{fmain1} is   a  cyclic cocycle  and has the following decomposition
\begin{equation}
\Psi_{2k}=\frac{i^{-k}}{(2k)!} \sum_{\langle (z,g)\rangle}\Phi_{2k;z,g}
\end{equation}
into the sum of localized cyclic cocycles \eqref{eq-char1}, where the summation is over all conjugacy classes in $\mathbb{C}^n\rtimes U(n)$ with nontrivial fixed point set. 
\end{corollary}

\section{Applications to Noncommutative Tori and Orbifolds}\label{sec5}


Here we specialize our spectral triple $(\A,\H,D)$ to subalgebras in $\A$ and obtain as corollaries of Theorem~\ref{th-main3} local index formulas on noncommutative tori of arbitrary dimension and noncommutative orbifolds.

\subsection*{The local index formula for noncommutative tori.}

Let $v_j\in\mathbb{C}^n$, $1\le j\le N$ be a collection of  vectors linearly independent over $\mathbb{Q}$.
They generate the lattice 
$$
 \left\{ \sum_j \ell_jv_j\;\Bigl|\; \ell_j\in\mathbb{Z}\Bigr.\right\}\subset\mathbb{C}^n
$$ 
isomorphic to $\mathbb{Z}^N$. We define the algebra $\A_v\subset \A$ of `functions on an $N$-dimensional  noncommutative torus' by
\begin{equation}
\label{eq-tori1}
 \A_v=\left\{\sum_\ell c_\ell T^{\ell_1}_{v_1}\cdots T^{\ell_N}_{v_N}\;|\; c_\ell\in\mathbb{C},\ell=(\ell_1,\ell_2,...,\ell_N),\ell_j\in\mathbb{Z}\right\}.
\end{equation} 
This is the algebra generated by the $N$ unitaries  
$$
T_{v_j}u(x)=e^{i(k_jx-a_jk_j/2)}u(x-a_j),\quad
\text{where }a_j=\re v_j,\; k_j=-\im v_j,
$$
acting on $\H=L^2(\mathbb{R}^n,\Lambda(\mathbb{C}^n))$ with the   commutation relations (cf.~\cite{Elliott1})
$$
 T_{v_k}T_{v_l}=e^{-i\im(v_k,v_l)}T_{v_l}T_{v_k}.
$$
Then we consider the spectral triple $(\A_v,\H,D)$, where $D$ was defined in~\eqref{euler2s}.
Corollary~\ref{cor1} implies that the Connes--Moscovici periodic cyclic cocycle  of this spectral triple decomposes into  cyclic cocycles  
\begin{equation}
\label{eq-indtor1}
 \Psi_{2k}(a_0,...,a_{2k})=\frac{i^{-k}}{(2k)!}\int_{\mathbb{C}^n} (a_0 da_1... da_{2k})(0)\wedge e^{-\omega}, \quad k\le n,  \quad a_j\in \A_v.
\end{equation} 
Here  $a_j$ are treated as elements in the differential graded algebra $\Omega^*$ (see~\eqref{eq-diff1} and~\eqref{eq-diff2}),  $(a_0 da_1... da_{2k})(0)\in \Lambda(\mathbb{C}^n)$ stands for the component corresponding to 
$\ell=0$ in \eqref{eq-tori1}, which is the only one with nontrivial fixed point set,
$\omega=dx\wedge dp$ is the symplectic form, while $\int_{{\mathbb{C}^n}}$ is the Berezin integral.

For $n=1$ Eq.~\eqref{eq-indtor1} coincides with the Connes cyclic cocycles in~\cite{Con1}, while for $n\ge 1$ this result is a refinement of the Riemann--Roch theorem on noncommutative tori of arbitrary dimension (see~\cite{MaRo20}).

\subsection*{The local index formula for noncommutative $\mathbb{Z}_4$-orbifolds.}

Choose complex numbers $z_1=k,z_2=ik$, $k>0$ and $g=i\in U(1)$. We define the square lattice $L=\{n_1z_1+n_2z_2\in\mathbb{C}\;|\; n_1,n_2\in\mathbb{Z}\}$
on which the group $\mathbb{Z}_4=\{i^\beta\;|\;\beta\in\mathbb{Z}\}$ acts by rotations.

To these elements, we associate the unitary operators $U=T_{z_1},V=T_{z_2}$, $R=R_{g}$:
$$
Uf(x)=f(x-k),\quad Vf(x)=e^{-ikx}f(x),\quad Ru(x)=(2\pi)^{-1/2}\int f(y)e^{-ixy}dy.
$$ 
We obtain the  commutation relations:
$$
VU=e^{ i\theta}UV,\quad RUR^{-1}=V,\quad RVR^{-1}=U^{-1}, \quad\text{where }\theta=- k^2.
$$ 
Hence, the algebra generated by $U$ and $V$ is just the noncommutative torus $\A_{\theta}$, while the algebra generated by $U,V,R$ is the  crossed product $\A_{\theta}\rtimes \mathbb{Z}_4$ with respect to the action of the generator 
of $\mathbb{Z}_4$ on  the generators $U,V\in \A_{\theta}$ as: 
$$
 U\longmapsto RUR^{-1}=V,\quad
 V\longmapsto RVR^{-1}=U^{-1}.
$$
This crossed product is known as a noncommutative orbifold for the group $\mathbb{Z}_4$ and was studied earlier in operator algebras and noncommutative geometry (see~\cite{FaWa1,Walt4,Walt2}).

It follows from the commutation relations  that elements   $a\in \A_{\theta}\rtimes \mathbb{Z}_4$ can be uniquely written as
$$
a=\sum_{(z,\alpha)\in L\times\mathbb{Z}_4}
a(z,\alpha)T_zR^\alpha .
$$ 

Consider the spectral triple $(\A_{\theta}\rtimes \mathbb{Z}_4,\H,D)$, which is the restriction of the spectral triple
$(\A,\H,D)$ in Section~\ref{sec3} to the subalgebra $\A_{\theta}\rtimes \mathbb{Z}_4\subset \A$.  Then Corollary~\ref{cor1} shows that the Connes--Moscovici periodic cyclic cocycle decomposes as a sum of cyclic cocycles   
$$
 \Phi_{2l;z,\alpha}\in HC^{2l}(\A_{\theta}\rtimes \mathbb{Z}_4),\quad l=0,1 
$$
for each $(z,\alpha)\in L\times\mathbb{Z}_4$. Let us describe these cocycles explicitly.

First, if $\alpha=0$, then the cocycles are nontrivial  only if $z=0$. In this case the fixed point set of the affine mapping $w\mapsto i^\alpha w+z=w$    is equal to $\mathbb{C}$ and  we have
$$
\Phi_{0;0,0}(a)=a(0,0),\quad \Phi_{2;0,0}(a_0,a_1,a_2)=\int_{\mathbb{C}} (a_0 da_1 d a_2)(0,0),
$$
where the exterior differential is that described above, while $\int_{\mathbb{C}}$ stands for the Berezin integral.

Second, if $\alpha\ne 0$, then the fixed point set of the affine mapping $w\mapsto i^\alpha w+ z$ is always zero dimensional; hence
the cocycles $\Phi_{2;z,\alpha}$ are trivial by  \eqref{fmain1}. Let us describe the trace $\Phi_{0;z,\alpha}$.    A direct computation shows that the conjugacy class $\langle(z,i^\alpha)\rangle\subset \mathbb{Z}^2\rtimes \mathbb{Z}_4$ is equal to 
\begin{equation}
\label{eq-conj1}
\langle(z,i^\alpha)\rangle=
\bigl(i^{\mathbb{Z}}z+L(1-i^{\alpha})\bigr)\times\{i^\alpha\}.  
\end{equation}
Hence, the cyclic cocycle is equal to 
$$
 \Phi_{0;z,\alpha}(f)=\sum_{z'\in i^{\mathbb{Z}}z+L(1-i^{\alpha})}\exp\left(\frac i 4 |z'|^2\ctg\frac{\pi\alpha}4\right)f(z',\alpha).
$$
A computation shows that there are actually eight different conjugacy classes in  \eqref{eq-conj1}, see the following table:

\begin{center}
\begin{tabular}{|c|c|l|}
 \hline 
 $z$ & $i^\alpha$ & conjugacy class  $\langle (z,i^\alpha)\rangle$  \\
 \hline
  $0$ & $1$ & $\{0\}\times \{1\}$ \\
  $0$ & $i$ & $k[(1-i)\mathbb{Z}+(1+i)\mathbb{Z}]\times \{i\}$ \\
  $k$ & $i$ & $k[1+(1-i)\mathbb{Z}+(1+i)\mathbb{Z}]\times \{i\}$ \\
  $0$ & $i^2$ & $k[2\mathbb{Z}+2i\mathbb{Z}]\times \{i^2\}$ \\
  $k$ & $i^2$ & $k[1+(1-i)\mathbb{Z}+(1+i)\mathbb{Z}]\times \{i^2\}$ \\
  $k(1+i)$ & $i^2$ & $k[1+i+2\mathbb{Z}+2i\mathbb{Z}]\times \{i^2\}$ \\
  $0$ & $i^3$ & $k[(1-i)\mathbb{Z}+(1+i)\mathbb{Z}]\times \{i^3\}$ \\
  $k$ & $i^3$ & $k[1+(1-i)\mathbb{Z}+(1+i)\mathbb{Z}]\times \{i^3\}$ \\  
  \hline
 \end{tabular}
\end{center}
 
Thus, the eight different traces $\Phi_{0,z,\alpha}$ coming from the decomposition of the Connes--Moscovici local index formula form a basis of the  eight-dimensional space of traces  on $\A_{\theta}\rtimes \mathbb{Z}_4$ (see~\cite{Walt4}).

\subsection*{The local index formula for noncommutative $\mathbb{Z}_6$-orbifolds.}

Choose complex numbers $z_1=k,z_2=\varepsilon k$, $g=\varepsilon\in U(1)$, where $k>0$ and $\varepsilon=e^{ \pi i/3}$. We consider the triangular lattice $L=\{n_1z_1+n_2z_2\in\mathbb{C}\;|\; n_1,n_2\in\mathbb{Z}\}$
with the group $\mathbb{Z}_6=\{\varepsilon^\beta\;|\;\beta\in\mathbb{Z}\}$ acting on $L$ by rotations.

Consider the unitary operators $U=T_{z_1},V=T_{z_2}$, $R=R_{g}$:
$$
Uf(x)=f(x-k),\quad 
Vf(x)=e^{i(-xk\sqrt{3}/2+k^2\sqrt{3}/8)}f(x-k/2), 
$$
$$
Ru(x)=\sqrt{\frac{1-\frac{i}{\sqrt{3}}}{2\pi}}
\int  \exp\left(  i\left( (x^2+y^2)\frac{1}{2\sqrt 3}-\frac{2xy}{\sqrt{3}} \right)        \right) u(y )dy .
$$  
We have the commutation relations:
$$
VU=e^{i \theta}UV,\quad RUR^{-1}=V,\quad RVR^{-1}=e^{-i \theta/2}U^{-1}V,\quad\text{where }\theta=-\frac{\sqrt 3}{2}k^2.
$$ 
Hence, the algebra generated by $U$ and $V$ is just the noncommutative torus $\A_{\theta}$, while the algebra generated by $U,V,R$ is the  crossed product $\A_{\theta}\rtimes \mathbb{Z}_6$ with respect to the action of the generator 
of $\mathbb{Z}_6$ on  the generators $U,V\in \A_{\theta}$ as: 
$$
 U\longmapsto RUR^{-1}=V,\quad
 V\longmapsto RVR^{-1}=e^{-i \theta/2}U^{-1}V.
$$
This crossed product is known as a noncommutative orbifold for the group $\mathbb{Z}_6$  and was studied earlier in operator algebras and noncommutative geometry (see~\cite{BuWa1,Walt3,Walt2,Walt1}).

It follows from the commutation relations  that elements   $f\in \A_{\theta}\rtimes \mathbb{Z}_6$ can be uniquely written as
$$
f=\sum_{(z,\alpha)\in L\times\mathbb{Z}_6} f(z,\alpha)T_zR^\alpha.
$$

Consider the spectral triple $(\A_{\theta}\rtimes \mathbb{Z}_6,\H,D)$, which is the restriction of the spectral triple
$(\A,\H,D)$ in Section~\ref{sec3} to the subalgebra $\A_{\theta}\rtimes \mathbb{Z}_6\subset \A$.  Then Corollary~\ref{cor1} shows that the Connes--Moscovici periodic cyclic cocycle decomposes as a sum of cyclic cocycles   
$$
 \Phi_{2l;z,\alpha}\in HC^{2l}(\A_{\theta}\rtimes \mathbb{Z}_6),\quad l=0,1, 
$$
for each $(z,\alpha)\in L\times\mathbb{Z}_6$. Let us describe these cocycles explicitly.

First, if $\alpha=0$, then the cocycles are nontrivial  only if $z=0$. In this case the fixed point set of the rotation   $w\mapsto \varepsilon^\alpha w $ is equal to $\mathbb{C}$ and  we have
$$
\Phi_{0;0,0}(a)=a(0,0),\quad \Phi_{2;0,0}(a_0,a_1,a_2)=\int_{\mathbb{C}} (a_0 da_1 d a_2)(0,0),
$$
where the exterior differential is that described above, and $\int_{\mathbb{C}}$ stands for the Berezin integral.

Second, if $\alpha\ne 0$, then the fixed point set of the affine mapping $w\mapsto \varepsilon^\alpha w+ z$ is of dimension zero and the 
cocycle $\Phi_{2;z,\alpha}$ is trivial by \eqref{fmain1}. 
Let us describe the trace $\Phi_{0;z,\alpha}$.    A direct computation shows that the conjugacy class $\langle(z,\varepsilon^\alpha)\rangle\subset L\rtimes \mathbb{Z}_6$ is equal to 
\begin{equation}
\label{eq-conj2}
\langle(z,\varepsilon^\alpha)\rangle=
[\varepsilon^{\mathbb{Z}}z+L(1-\varepsilon^{\alpha})]\times\{\varepsilon^\alpha\}.
\end{equation}
Hence, the trace is equal to 
$$
 \Phi_{0;z,\alpha}(f)=\sum_{z'\in \varepsilon^{\mathbb{Z}}z+L(1-\varepsilon^{\alpha})}\exp\left(\frac i 4 |z'|^2\ctg\frac{\pi\alpha}6\right)f(z',\alpha).
$$ 
A computation shows that there are actually nine different conjugacy classes in  \eqref{eq-conj2}, see the following table:

\begin{center}
\begin{tabular}{|c|c|l|}
 \hline 
 $z$ & $\varepsilon^\alpha$ &  conjugacy class $\langle (z,\varepsilon^\alpha)\rangle$  \\
 \hline
  $0$ & $1$ & $\{0\}\times \{1\}$ \\
  $0$ & $\varepsilon$ & $
  L\times \{\varepsilon\}$ \\
  $0$ & $\varepsilon^2$ & $k[(\varepsilon^2-1)\mathbb{Z}+(\varepsilon+1)\mathbb{Z}]\times \{\varepsilon^2\}$ \\
  $k$ & $\varepsilon^2$ & $(L\setminus k[(\varepsilon^2-1)\mathbb{Z}+(\varepsilon+1)\mathbb{Z}])\times \{\varepsilon^2\}$ \\
  $0$ & $\varepsilon^3$ & $2L\times \{\varepsilon^3\}$ \\
  $k$ & $\varepsilon^3$ & $(L\setminus 2L)\times \{\varepsilon^3\}$ \\
  $0$ & $\varepsilon^4$ & $k[ (\varepsilon+1)\mathbb{Z}+(\varepsilon^2+\varepsilon) \mathbb{Z}]\times \{\varepsilon^4\}$ \\
  $k$ & $\varepsilon^4$ & $(L\setminus k[ (\varepsilon+1)\mathbb{Z}+(\varepsilon^2+\varepsilon) \mathbb{Z}])\times \{\varepsilon^4\}$ \\
   $0$ & $\varepsilon^5$ & $
   L\times \{\varepsilon^5\}$ \\
  \hline
 \end{tabular}
\end{center}

Thus, the nine different traces $\Phi_{0,z,\alpha}$ coming from the decomposition of the Connes--Moscovici local index formula form a basis of the  nine-dimensional space of traces  on $\A_{\theta}\rtimes \mathbb{Z}_6$ (see~\cite{BuWa1}).

\section{Equivariant Zeta Functions for the Affine Metaplectic Group}\label{sec:zeta} 

Let $A\in\Psi (\mathbb{R}^n)$ be a Shubin type pseudifferential  operator of order $\ord A$, $g\in U(n), w\in \mathbb{C}^n$. For $z\in \mathbb C$ with $\re z$ sufficiently large consider the zeta function
\begin{equation}\label{eq-zeta2}
 \zeta_{A,g,w}(z)=\tr(R_gT_wA H^{-z}). 
\end{equation}

\begin{theorem}\label{zeta1}
The zeta function \eqref{eq-zeta2} has the following properties:
\begin{enumerate}
\item It is well defined and holomorphic in the half-plane $\ord A-2\re z<-2n$;
\item It has a meromorphic continuation to $\mathbb{C}$ with possibly simple poles at the points
$$
z=\dim \mathbb{C}^n_g +(\ord A-j)/2, \qquad j\in \mathbb{Z}_+,
$$  
where $ \mathbb{C}^n_g  $ is the fixed point set of $g:\mathbb{C}^n\to\mathbb{C}^n$. Moreover, if the fixed point set of the affine mapping $\mathbb{C}^n\to\mathbb{C}^n, v\mapsto gv+w$ is empty, then the zeta function has no poles.
\end{enumerate}
\end{theorem}

\begin{proof}
1. The operator $AH^{-\re z}$ is a Shubin type pseudodifferential operator  of order $\le \ord A-2\re z< -2n$ by the assumption. Hence, it is of trace class. Thus, the zeta function is well defined and holomorphic in $z$, since $R_g,T_w,H^{-i\im z}$ are bounded operators.

2. Let us now show that the zeta function has a meromorphic continuation to $\mathbb{C}$. Without loss of generality we can assume that $g$ is a diagonal matrix. Indeed, if this is not the case, then we have $g=ug_0u^{-1}$, where $u$  is  unitary, while $g_0$ is diagonal and unitary. Hence:
\begin{multline}
 \zeta_{A,g,w}(z)=\tr(R_gT_wA H^{-z})=\tr(R_uR_{g_0}R_u^{-1}T_wA H^{-z})\\
 =\tr( R_{g_0}R_u^{-1}T_w(R_u R_u^{-1})A (R_u R_u^{-1})H^{-z}R_u)\\
 =\tr( R_{g_0}(R_u^{-1}T_wR_u)(R_u^{-1} A  R_u)  (R_u^{-1}H^{-z}R_u))\\
 =
  \tr( R_{g_0}T_{w'}A'   H^{-z} )= \zeta_{A',g',w'}(z)
\end{multline}
Here $A'=R_u^{-1} A  R_u$ is a Shubin type operator by Egorov's theorem, $R_u^{-1}T_wR_u= T_{w'}$, where $w'=u^{-1}w$, and we used the fact that $H$ commutes with $R_u$.

Let now $w=(w_1,...,w_n)= a-ik$, where $a,k\in \mathbb{R}^n$, and consider diagonal element
\begin{equation}
\label{eq-diag2}
g={\rm diag}\Bigl(\underbrace{e^{i\varphi_1},...,e^{i\varphi_{m_1}}}_{m_1},\underbrace{i,...,i}_{m_2},\underbrace{-i,...,-i}_{m_3},\underbrace{-1,...,-1}_{m_4},\underbrace{1,...,1}_{m_5}\Bigr),
\end{equation}
where $\varphi_j\notin \pi\mathbb{Z}/2$ and   $m_5=\dim(\mathbb{C}^n)^g$.   
For later purposes we also let $\varphi_j=\pi/2$ for 
$j=m_1+1, \ldots, m_1+m_2$ and $\varphi_j = 3\pi/2$ for $j= m_1+m_2+1,\ldots, m_1+m_2+m_3$. 

The Schwartz kernel of $A H^{-z} $ is written as 
\begin{equation}
K_{AH^{-z}}(x,x')=\int e^{i(x-x')p}q(x',p;z)dp,
\end{equation}
where  $q(x',p;z)$ is a classical symbol of order $\ord A-2z$.
Then the  Schwartz kernel of $T_wA H^{-z} $ is equal to 
\begin{equation}
K_{T_wAH^{-z}}(x,x')=Const\int e^{i((x-a-x')p+kx)}q(x',p;z)dp.
\end{equation}
Since $g$ is a diagonal matrix, the operator $R_g$ is a product of fractional Fourier transforms in the variables $x_1,..,x_n$ and its Schwartz kernel is given by the Mehler formula and, hence, the Schwartz kernel of $R_gT_wA H^{-z} $ is equal to 
\begin{equation}
K_{R_gT_wAH^{-z}}(x,x')=Const\int  e^{i\phi_1}q(x',p;z)dp  dx'',\quad x''=(x''_1,...,x''_{m_1+m_2+m_3}),
\end{equation}
where 
\begin{multline*}
\phi_1=\sum_{j=1}^{m_1+m_2+m_3}\left(-\frac{x_jx_j''}{\sin\varphi_j}+\frac{\ctg\varphi_j}2(x_j^2+{x_j''}^2)
+(x_j''-a_j-x_j')p_j+k_jx_j''\right)\\
+\sum_{j=m_1+m_2+m_3+1}^{n-m_5}((-x_j -a_j-x_j')p_j-k_jx_j ) +\sum_{j=n-m_5+1}^n((x_j -a_j-x_j')p_j+k_jx_j ).
\end{multline*}
Hence, the zeta function is equal to
\begin{equation}\label{eq-fixa1}
\zeta_{A,g,w}(z)=\int K_{R_gT_wAH^{-z}}(x,x )dx= Const\int  e^{i\phi_2}q(x,p;z)dp  dx''dx,
\end{equation}
where
\begin{multline*}
\phi_2=\sum_{j=1}^{m_1+m_2+m_3}\left(-\frac{x_jx_j''}{\sin\varphi_j}+\frac{\ctg\varphi_j}2(x_j^2+{x_j''}^2)
+(x_j''-a_j-x_j)p_j+k_jx_j''\right)\\
-\sum_{j=m_1+m_2+m_3+1}^{n-m_5}((2x_j +a_j)p_j+k_jx_j )+\sum_{j=n-m_5+1}^n( -a_j p_j+k_jx_j ).
\end{multline*}
Note that if some $a_j\ne 0$ (or $k_j\ne 0$) for $j>n-m_5$,\footnote{
This condition is equivalent to the condition that the   affine mapping $z\mapsto gz+w$ has no fixed points.} then integrations by parts in \eqref{eq-fixa1}
with respect to $p_j$ (respectively $x_j$) show that $\zeta_{A,g,w}(z)$ can also be represented by an integral, where
the pseudodifferential symbol has very negative order. This proves that in this case the zeta function is in fact an entire function in $\mathbb{C}$.

Thus, below we suppose that $a_j=k_j=0$ for all $j>n-m_5.$
Let us now compute the Gaussian integral over $x''$ in \eqref{eq-fixa1}:
\begin{multline*}
\int  \exp \Biggl(i 
\left( \frac{\ctg\varphi_j}2 {x_j''}^2 
+ x_j''\left( p_j+k_j -\frac{x_j }{\sin\varphi_j}\right)\right)\Biggr)dx''_j\\
=
\left\{
\begin{array}{ll}
\displaystyle Const\exp \Bigl(-\frac i 2 
\tg\varphi_j\Bigl(p_j+k_j-\frac{x_j }{\sin\varphi_j}\Bigr)^2\Bigr), 
& \text{if }\varphi_j\notin \pi\mathbb{Z}/2,\\
\displaystyle Const \;\;\delta\left(p_j+k_j\mp x_j\right), & \text{if }\varphi_j=\pm\pi /2,\\
\end{array}
\right.
\end{multline*}
and obtain
\begin{equation}\label{eq-fixa2}
\zeta_{A,g,w}(z)=Const\int  e^{i\phi_3(x,p')}q(x,p',p'';z)dxdp'.  
\end{equation}
Here we decomposed $p$ as follows: $p=(p',p'')$, where $p'=(p_1,...,p_{m_1},p_{m_1+m_2+m_3+1},...,p_n)$ and $p''=(p_{m_1+1},...,p_{m_1+m_2+m_3})$. Note that the integration of the $\delta$-functions gives us $p_j=\pm x_j-k_j$ for all $j=m_1+1,...,m_1+m_2+m_3$.  The phase function in~\eqref{eq-fixa2} is equal to
\begin{multline}\label{phase3}
\phi_3(x,p')=
\\ \sum_{j=1}^{m_1}
     \Bigl(
        x_j^2\Bigl(\frac{\ctg\varphi_j}2-\frac{1}{ \sin 2\varphi_j }\Bigr)
       +x_jp_j\Bigl(\frac{1}{\cos\varphi_j}-1 \Bigr)
       -p_j^2\frac{\tg\varphi_j}2  
       +x_j\frac{k_j}{\cos\varphi_j}
       -p_j\left( a_j+k_j\tg\varphi_j\right)-\frac{\tg\varphi_j k_j^2}{2}
      \Bigr)  \\
-\sum_{j=m_1+1}^{m_1+m_2+m_3}(x_j+a_j)\left(\frac{x_j}{\sin\varphi_j}-k_j\right)       
-\sum_{j=m_1+m_2+m_3+1}^{n-m_5}( 2x_jp_j +a_j p_j+k_jx_j ).
\end{multline}
A change of variables $(x,p')=Bv+b$, where $B\in O(\nu)$, $\nu=2n-m_2-m_3$, is an orthogonal matrix and $v,b\in\mathbb{R}^{\nu}$, makes the phase function $\phi_3$ quadratic in $v$ plus a constant:
\begin{equation}\label{eq-phase6}
\phi_3(x,p')=\sum_{j=1}^{\nu} \lambda_j v_j^2+Const.
\end{equation}
Note that $B$ and $b$ depend only on $g$ and $w$. We introduce spherical coordinates $v=r\theta,$ where   $r\ge 0$ and $\theta\in\mathbb{S}^{\nu-1}$ in \eqref{eq-fixa2}, and obtain
\begin{multline}\label{eq-fixa3}
\int  e^{i\phi_3(x,p')}q(x,p',p'';z)dxdp'=Const 
 \int_0^\infty  \Bigl(  \int_{\mathbb{S}^{\nu-1}}  \exp\Bigl(i\sum_{j=1}^{\nu} \lambda_j v_j^2 \Bigr)q(Br\theta+b,p'';z)d\theta\Bigr)r^{\nu-1}dr\\
 \equiv Const 
 \int_0^\infty  c(r;z) r^{\nu-1}dr
\end{multline}
where
\begin{equation}\label{eq-c4}
   c(r;z)=\int_{\mathbb{S}^{\nu-1}}  \exp\Bigl(i\sum_{j=1}^{\nu} \lambda_j v_j^2 \Bigr)q(Br\theta+b,p'';z)d\theta.
\end{equation}
The asymptotics of $c(r;z)$ as $r\to\infty$ can be computed by the stationary phase formula.   To state it, we denote by $\{\mu_l\}$ all the different numbers $\lambda_1,...,\lambda_{\nu}$ in \eqref{eq-phase6},
and let $\{\varkappa_l\}$ be their multiplicities.

\begin{lemma}
We have an asymptotic expansion as $r\to\infty$:
\begin{equation}\label{eq-asympt5}
  c(r;z)\sim r^{\ord A-2z- \nu}\sum_l r^{\varkappa_l}e^{ir^2\mu_l}\sum_{j\ge 0} c_{lj}(z)r^{-j},
\end{equation}
where the coefficients $c_{lj}(z)$ are entire functions of $z$.
\end{lemma}
\begin{proof}
One shows that the stationary points of the phase function in \eqref{eq-c4}  are just the unit length eigenvectors
of the diagonal matrix ${\rm diag}(\lambda_1,...,\lambda_{\nu})$. Hence, the set of stationary points is just the disjoint union of spheres $\mathbb{S}^{\varkappa_l-1}\subset \mathbb{S}^{\nu-1}$ over all distinct eigenvalues. Moreover, these critical submanifolds are nondegenerate.
Thus, application of the stationary phase formula (with large parameter equal to $r^2$) together with the fact that $a(x,p;z)$ is a classical symbol of order $\ord A-2z$, gives us the desired asymptotic expansion:
\begin{equation}\label{eq-asympt6}
  c(r;z)\sim r^{\ord A-2z}\sum_l r^{\varkappa_l-\nu} e^{ir^2\mu_l}\sum_{j\ge 0} c_{lj}(z)r^{-j}.
\end{equation}
\end{proof}
 
We now substitute \eqref{eq-asympt5} in \eqref{eq-fixa3} and obtain that modulo entire functions   $\zeta_{A,g,w}(z)$
in \eqref{eq-fixa1} is equal to the series:
\begin{equation}\label{eq-7}
\zeta_{A,g,w}(z)\equiv Const\sum_l \sum_{j\ge 0}c_{lj}(z) \int_1^\infty r^{\ord A-2z+\varkappa_l-1-j}e^{ir^2\mu_l}dr.
\end{equation}
Integration by parts shows that the integral
$$
\int_1^\infty r^{\ord A-2z+\varkappa_l-1-j}e^{ir^2\mu_l}dr
$$
is an entire function of $z$ unless $\mu_l=0$. Suppose for definiteness that $\mu_1=0$.
Thus, we obtain the following equality modulo entire functions:
\begin{equation}\label{eq-8}
\zeta_{A,g,w}(z)
 \equiv  Const \sum_{j\ge 0}c_{1j}(z) \int_1^\infty r^{\ord A-2z+\varkappa_1-1-j} dr
 \equiv Const \sum_{j\ge 0}\frac{-c_{1j}(z)/2}{z-(\ord A+\varkappa_1-j)/2}.
\end{equation}
It remains to note that $\varkappa_1=2\dim  \mathbb{C}^n_g$ is the real dimension of the fixed point set of $g\in U(n)$ (this follows from \eqref{phase3} and \eqref{eq-diag2}).

This completes the proof of Theorem~\ref{zeta1}.
\end{proof}




\begin{thebibliography}{10}


\bibitem{Arn1}
V.~I. Arnold.
\newblock {\em Mathematical Methods of Classical Mechanics}, volume~60 of {\em
  Graduate Texts in Mathematics}.
\newblock Springer--Verlag, Berlin--Heidelberg--New York, second edition, 1989.

%
\bibitem{Azmi00}
F.~Azmi, The equivariant Dirac cyclic cocycle, {\em Rocky Mountain J. Math.} 30(4):1171--1206, 2000. 


\bibitem{BCHL1}
Ch.~B\"onicke, S.~Chakraborty, Z.~He,  H.-C.~Liao. 
\newblock Isomorphism and Morita equivalence classes for crossed products of irrational rotation algebras by cyclic subgroups of $SL_2(\mathbb{Z})$. 
\newblock {\em J. Funct. Anal.}, 275(11):3208–-3243 (2018).

\bibitem{BuWa1}
J.~Buck and S.~Walters.
\newblock Connes-Chern characters of hexic and cubic modules.
\newblock {\em  J. Operator Theory}, 57(1):35--65, 2007.

\bibitem{BuMa1}
A.~Bultheel and H.~Mart\'{\i}nez-Sulbaran.
\newblock Recent developments in the theory of the fractional {F}ourier and
  linear canonical transforms.
\newblock {\em Bull. Belg. Math. Soc. Simon Stevin}, 13(5):971--1005, 2006.

\bibitem{CPRS1}
A.~L. Carey, J.~Phillips, A.~Rennie, and F.~A. Sukochev.
\newblock The local index formula in semifinite von {N}eumann algebras. {II}.
  {T}he even case.
\newblock {\em Adv. Math.}, 202(2):517--554, 2006.

\bibitem{CPRS2}
A.~L. Carey, J.~Phillips, A.~Rennie, and F.~A. Sukochev.
\newblock The local index formula in noncommutative geometry revisited. Noncommutative geometry and physics. 3, 3--36, Keio COE Lect. Ser. Math. Sci., 1, World Sci. Publ., Hackensack, NJ, 2013.

\bibitem{ChLu1}
S.~Chakraborty, F.~Luef. 
\newblock Metaplectic transformations and finite group actions on noncommutative tori. 
\newblock {\em J. Operator Theory}, 82(1):147--172 (2019).

\bibitem{CheHu1}
Sh. Chern and X.~Hu.
\newblock Equivariant {C}hern character for the invariant {D}irac operator.
\newblock {\em Michigan Math. J.}, 44(3):451--473, 1997.

\bibitem{Con1}
A.~Connes.
\newblock {\em Noncommutative geometry}.
\newblock Academic Press Inc., San Diego, CA, 1994.

\bibitem{CoMo1}
A.~Connes and H.~Moscovici.
\newblock The local index formula in noncommutative geometry.
\newblock {\em Geom. Funct. Anal.}, 5(2):174--243, 1995.

\bibitem{CoMo2}
A.~Connes and H.~Moscovici.
\newblock Type {III} and spectral triples.
\newblock In {\em Traces in number theory, geometry and quantum fields},
  Aspects Math., E38, pages 57--71. Friedr. Vieweg, Wiesbaden, 2008.

\bibitem{CoMo3}
A.~Connes and H.~Moscovici.
\newblock Hopf algebras, cyclic cohomology and the transverse index theorem.
\newblock {\em Comm. Math. Phys.}, 198(1):199--246, 1998.


\bibitem{Dab1}
L.~D\c{a}browski.
\newblock The local index formula for quantum {${\rm SU}(2)$}.
\newblock In {\em Traces in number theory, geometry and quantum fields},
  Aspects Math., E38, pages 99--110. Friedr. Vieweg, Wiesbaden, 2008.

\bibitem{Walt2}
S.~Echterhoff, W.~L{\"u}ck, N.~C. Phillips, and S.~Walters.
\newblock The structure of crossed products of irrational rotation algebras by
  finite subgroups of {${\rm SL}_2(\mathbb Z)$}.
\newblock {\em J. Reine Angew. Math.}, 639:173--221, 2010.

\bibitem{Elliott1}
G.~A.~Elliott.
\newblock On the {$K$}-theory of the {$C^{\ast} $}-algebra generated
              by a projective representation of a torsion-free discrete abelian group.
\newblock In {\em Operator algebras and group representations, {V}ol. {I} ({N}eptun, 1980)},
  Monogr. Stud. Math., 17, pages 157--184. Pitman, Boston, MA, 1984.
     

\bibitem{Fol1}
G.~B. Folland.
\newblock {\em Harmonic analysis in phase space}, volume 122 of {\em Annals of
  Mathematics Studies}.
\newblock Princeton University Press, Princeton, NJ, 1989.

\bibitem{Get1}
E.~Getzler.
\newblock Pseudodifferential operators on supermanifolds and the
  {Atiyah--Singer} index theorem.
\newblock {\em Commun. Math. Phys.}, 92:163--178, 1983.

\bibitem{FaWa1}
C.~Farsi and N.~Watling.
\newblock Quartic algebras.
\newblock {\em Canad. J. Math.}, 44(6):1167--1191, 1992.

\bibitem{FaLuTa}
F.~Fathizadeh, F.~Luef, and J.~Tao. 
A twisted local index formula for curved noncommutative two tori. 
Preprint arXiv:1904.03810, 2019.

\bibitem{FGD1}
Ph. Flajolet, X.~Gourdon, and Ph. Dumas.
\newblock Mellin transforms and asymptotics: harmonic sums.
\newblock volume 144, pages 3--58, 1995.
\newblock Special volume on mathematical analysis of algorithms.
 
\bibitem{GaJuDo} 
Li Gao, M.~Junge, and E.~McDonald. 
Quantum Euclidean spaces with noncommutative derivatives. 
{\em J. Noncommut. Geom.} 16(1):153?213, 2022.

\bibitem{GKN}
A. Gorokhovsky, N. de Kleijn and R. Nest. 
\newblock Equivariant algebraic index theorem. 
\newblock {\em J. Inst. Math. Jussieu}   20(3):929--955, 2021.

\bibitem{GoErp1}
A.~Gorokhovsky and E.~van Erp.
\newblock Index theory and noncommutative geometry: a survey. 
{\em Advances in noncommutative geometry--on the occasion of Alain Connes' 70th birthday}, 421--462, Springer, Cham, 2019. 

\bibitem{deG1}
M.~de~Gosson.
\newblock {\em Symplectic geometry and quantum mechanics}, volume 166 of {\em
  Operator Theory: Advances and Applications}.
\newblock Birkh\"{a}user Verlag, Basel, 2006.

\bibitem{GS1}
G.~Grubb and R.~T. Seeley.
\newblock Weakly parametric pseudodifferential operators and
  {A}tiyah-{P}atodi-{S}inger boundary problems.
\newblock {\em Invent. Math.}, 121(3):481--529, 1995.

\bibitem{HKT1}
N.~Higson, G.~Kasparov, and J.~Trout.
\newblock {A Bott periodicity theorem for infinite dimensional Euclidean
  space.}
\newblock {\em Adv. Math.}, 135(1):1--40, 1998.

\bibitem{Hig04}
N.~Higson.
\newblock The local index formula in noncommutative geometry.
\newblock In {\em Contemporary developments in algebraic $K$-theory},  {\em ICTP Lect. Notes, XV}, pages 443--536. 
Abdus Salam Int. Cent. Theoret. Phys., Trieste, 2004.

\bibitem{Hig6}
N.~Higson.
\newblock The residue index theorem of {C}onnes and {M}oscovici.
\newblock In {\em Surveys in noncommutative geometry}, volume~6 of {\em Clay
  Math. Proc.}, pages 71--126. Amer. Math. Soc., Providence, RI, 2006.

\bibitem{H95} 
L.~H\"ormander. 
\newblock Symplectic classification of quadratic forms, and general Mehler formulas. 
\newblock {\em Math. Z.} 219:413--449, 1995. 

\bibitem{Ler8}
J.~Leray.
\newblock {\em Analyse Lagrangienne et {M}{\'e}canique {Q}uantique}.
\newblock IRMA, Strasbourg, 1978.

\bibitem{Les3}
J.-M. Lescure.
\newblock Triplets spectraux pour les vari\'{e}t\'{e}s \`a singularit\'{e}
  conique isol\'{e}e.
\newblock {\em Bull. Soc. Math. France}, 129(4):593--623, 2001.

\bibitem{MaRo20}
V.~Mathai and J.~Rosenberg.
\newblock The Riemann-Roch theorem on higher dimensional complex noncommutative tori, {\em J. Geom. Phys}. 147, 103534, 9 (2020).

\bibitem{Mos2}
H.~Moscovici.
\newblock Local index formula and twisted spectral triples.
\newblock In {\em Quanta of maths}, volume~11 of {\em Clay Math. Proc.}, pages
  465--500. Amer. Math. Soc., Providence, RI, 2010.

\bibitem{NeTu1}
S.~Neshveyev and L.~Tuset.
\newblock A local index formula for the quantum sphere.
\newblock {\em Comm. Math. Phys.}, 254(2):323--341, 2005.

\bibitem{Pon4}
R.~Ponge.
\newblock A new short proof of the local index formula and some of its
   applications.
\newblock {\em Comm. Math. Phys},  241(2-3):215--234, 2003.

\bibitem{Pon20}
R.~Ponge. Noncommutative residue and canonical trace on noncommutative tori. Uniqueness results.
{\em SIGMA Symmetry Integrability Geom. Methods Appl.}  16,  Paper No. 061, 31 pp., 2020.

\bibitem{PoWa2}
R.~Ponge and H.~Wang.
\newblock Noncommutative geometry and conformal geometry. {I}. {L}ocal index
  formula and conformal invariants.
\newblock {\em J. Noncommut. Geom.}, 12(4):1573--1639, 2018.

\bibitem{SaSchSt4}
A.~Savin, E.~Schrohe, and B.~Sternin.
\newblock Elliptic operators associated with groups of quantized canonical
  transformations.
\newblock {\em {B}ull. {S}ci. {M}ath.}, 155:141--167, 2019.

\bibitem{SaSch1}
A.~Savin and E.~Schrohe.
\newblock Analytic and algebraic indices of elliptic operators associated with
  discrete groups of quantized canonical transformations.
\newblock {\em J. Funct. Anal.}, 278(5):108400, 45, 2020. 

\bibitem{SaSc2}
A.~Savin and E.~Schrohe. 
\newblock An index formula for groups of isometric linear canonical transformations.
\newblock 
{\em Doc. Math.}, 27:983-1013, 2022. 
 
\bibitem{Shu1}
M.~A. Shubin.
\newblock {\em Pseudodifferential Operators and Spectral Theory}.
\newblock Springer--Verlag, Berlin--Heidelberg, 1985.

\bibitem{SDLSV1}
W.~van Suijlekom, L.~D\c{a}browski, G.~Landi, A.~Sitarz, and J.~C. V\'{a}rilly.
\newblock The local index formula for {${\rm SU}_q(2)$}.
\newblock {\em $K$-Theory}, 35(3-4):375--394, 2005.

\bibitem{Walt4}
S.~G. Walters.
\newblock Chern characters of {F}ourier modules.
\newblock {\em Canad. J. Math.}, 52(3):633--672, 2000.

\bibitem{Walt3}
S.~Walters.
\newblock Periodic integral transforms and {$C^*$}-algebras.
\newblock {\em C. R. Math. Acad. Sci. Soc. R. Can.}, 26(2):55--61, 2004.


\bibitem{Walt1}
S.~Walters.
\newblock Toroidal orbifolds of {$\mathbb{Z}_3$} and {$\mathbb{Z}_6$} symmetries of
  noncommutative tori.
\newblock {\em Nuclear Phys. B}, 894:496--526, 2015.

\bibitem{Wod2}
M.~Wodzicki.
\newblock Noncommutative residue. I. Fundamentals. 
\newblock {\em Lecture Notes in Math.}, 1289. Berlin, New York: Springer-Verlag, pp. 320--399, 1987.

\end{thebibliography}

\end{document}